\documentclass[a4paper,11pt]{article}
\usepackage{amssymb, amsmath, amsthm, tikz}
\usepackage{fullpage}
\usepackage{hyperref}
\usepackage{enumerate}

\newtheorem{theorem}{Theorem}[section]
\newtheorem{lemma}[theorem]{Lemma}
\newtheorem{cor}[theorem]{Corollary}
\newtheorem{conj}[theorem]{Conjecture}
\newtheorem{claim}[theorem]{Claim}

\theoremstyle{definition}

\newtheorem{defn}[theorem]{Definition}

\newenvironment{poc}{\begin{proof}[Proof of claim]}{\end{proof}}

\newcommand{\bN}{\ensuremath{\mathbb{N}}}
\newcommand{\ie}{i.e.\ }
\newcommand*{\abs}[1]{\lvert#1\rvert}
\newcommand*{\ceil}[1]{\lceil#1\rceil}

\title{Ramsey numbers of cycles versus general graphs}
\author{John Haslegrave\thanks{Mathematical Institute, Andrew Wiles Building, University of Oxford, Oxford OX2 6GG, UK (\texttt{j.haslegrave@cantab.net}). Supported
by the UK Research and Innovation Future Leaders Fellowship MR/S016325/1.}
\and Joseph Hyde\thanks{Mathematics and Statistics, University of Victoria,, Victoria, B.C., V8W 2Y2, Canada (\texttt{josephhyde@uvic.ca}). Supported
by the UK Research and Innovation Future Leaders Fellowship MR/S016325/1 and ERC Advanced Grant 101020255.}
\and Jaehoon Kim \thanks{Department of Mathematical Sciences, KAIST, South Korea (\texttt{jaehoon.kim@kaist.ac.kr}). Supported by the POSCO Science Fellowship of POSCO TJ Park Foundation and by the KAIX Challenge program of KAIST Advanced Institute for Science-X.}
\and Hong Liu
\thanks{Extremal Combinatorics and Probability Group (ECOPRO), Institute for Basic Science (IBS), Daejeon, South Korea (\texttt{hongliu@ibs.re.kr}). Supported by the Institute for Basic Science (IBS-R029-C4) and the UK Research and Innovation Future Leaders Fellowship MR/S016325/1.}
}

\begin{document}
\maketitle

\begin{abstract}
The Ramsey number $R(F,H)$ is the minimum number $N$ such that any $N$-vertex graph either contains a copy of $F$ or its complement contains $H$. Burr in 1981 proved a pleasingly general result that for any graph $H$, provided $n$ is sufficiently large, a natural lower bound construction gives the correct Ramsey number involving cycles: $R(C_n,H)=(n-1)(\chi(H)-1)+\sigma(H)$, where $\sigma(H)$ is the minimum possible size of a colour class in a $\chi(H)$-colouring of $H$. Allen, Brightwell and Skokan conjectured that the same should be true already when $n\geq \abs{H}\chi(H)$. 

We improve this 40-year-old result of Burr by giving quantitative bounds of the form $n\geq C\abs{H}\log^4\chi(H)$, which is optimal up to the logarithmic factor. In particular, this proves a strengthening of the Allen-Brightwell-Skokan conjecture for all graphs $H$ with large chromatic number. 
\end{abstract}

\section{Introduction}

For any pair of graphs $F$ and $H$, Ramsey~\cite{R30} famously proved in 1930 that  there exists a number $R(F,H)$ such that given any graph $G$ on at least $R(F,H)$ vertices either $F \subseteq G$ or $H \subseteq \overline{G}$. Determining $R(F,H)$ exactly for every pair of graphs $F,H$ is a notoriously difficult problem in combinatorics. 
Indeed, $R(K_n, K_n)$ is not even known exactly for $n = 5$.
However, for some pairs of graphs, particularly when $F$ and/or $H$ are certain sparse graphs, $R(F,H)$ is known. We require some notation.
For a graph $H$, define the chromatic number $\chi(H)$ of $H$ to be the smallest number of colours in a proper colouring of $H$, that is, a colouring where no two adjacent vertices have the same colour.
Further, let $\sigma(H)$ be the minimum possible size of a colour class in a $\chi(H)$-colouring of $H$. 

Building on observations of Erd\H{o}s~\cite{E47} and Chv\'{a}tal and Harary~\cite{CH72}, Burr~\cite{B81} constructed the following lower bound for $R(F,H)$ when $F$ is a connected graph with $\abs{F} \geq \sigma(H)$: 
\begin{equation}\label{lem:burr}
R(F,H) \geq (\abs{F} - 1)(\chi(H) - 1) + \sigma(H).
\end{equation}
The construction proving this bound is the graph $G$ on $(\abs{F} - 1)(\chi(H) - 1) + \sigma(H) - 1$ vertices consisting of $\chi(H) - 1$ disjoint cliques of size $\abs{F}-1$ and an additional disjoint clique of size $\sigma(H) - 1$. 
Clearly, since $F$ is connected, $F \nsubseteq G$. 
Since $\overline{G}$ is the complete $\chi(H)$-partite graph with $\chi(H)-1$ vertex sets of size $\abs{F} - 1$ and one vertex set of size $\sigma(H) - 1$, and the vertex set of size $\sigma(H) - 1$ cannot completely contain any colour class of $H$, we have $H \nsubseteq \overline{G}$. In what follows we will often say a graph $G$ is \emph{$H$-free} to mean $H \nsubseteq G$.

Although the bound in~\eqref{lem:burr} is very general, for some pairs of graphs it is extremely far from the truth.
Indeed, Erd\H{o}s \cite{E47} proved that $R(K_k, K_k) \geq \Omega(2^{k/2})$, whereas \eqref{lem:burr} only gives the much smaller lower bound of $(k-1)^2 + 1$. 
For some pairs of graphs however the bound in~\eqref{lem:burr} is tight. For graphs $F$ and $H$ satisfying this lower bound 
$R(F,H)=(\chi(H)-1)(\abs{F}-1)+\sigma(H)$, Burr and Erd\H{o}s~\cite{BE83} coined the expression `$F$ is \emph{$H$-good}'. The study of so-called \emph{Ramsey goodness}, after being initiated by Burr and Erd\H{o}s in 1983~\cite{BE83}, has attracted considerable interest.
Prior to Burr's observation~\cite{B81}, already Erd\H{o}s~\cite{E47} had proved that the path on $n$ vertices $P_n$ was $K_k$-good; Gerensc\'{e}r and Gy\'{a}rf\'{a}s~\cite{GG67} proved that for $n \geq k$ the path $P_n$ is $P_k$-good; and Chv\'{a}tal~\cite{C77} proved every tree $T$ is $K_k$-good. The result of Chvatal can be in fact viewed as a generalisation of Tur\'{a}n's theorem, which states that the complete balanced $(k-1)$-partite graph on $N$ vertices, the so-called \emph{Tur\'{a}n graph}, has the maximum number of edges amongst $K_k$-free graphs.
Indeed, Tur\'{a}n's theorem is equivalent to the statement `$S_n$ is $K_k$-good', where $S_n$ is the $n$-vertex star with $n-1$ leaves. To see this, let $N = (n-1)(k-1)$ and $T$ be the complement of this Tur\'{a}n graph. Then $T$ is precisely the construction in~\eqref{lem:burr} when $(F,H) = (S_n, K_k)$. Moreover, $T$ is the unique graph on $(n-1)(k-1)$ vertices without $S_n \subseteq T$ or $K_k \subseteq \overline{T}$. Add a vertex to $T$ and denote the resulting graph by $T + v$. If $v$ neighbours any vertex in $T$, then $S_n \subseteq T + v$. Otherwise, $K_k \subseteq \overline{T + v}$. That is, $R(S_n, K_k) = (n-1)(k-1) + \sigma(K_k) = (n-1)(k-1) + 1$. This connection to Tur\'an's theorem highlights how Ramsey goodness results can generalise other results in graph theory. 
See \cite{FHW, LL21, M21, PS17, PS20} and their references for more recent progress in the area of Ramsey goodness, as well as the survey of Conlon, Fox and Sudakov~\cite[Section 2.5]{CFS15}.

\smallskip

In this paper we are specifically interested in when $C_n$, the $n$-vertex cycle, is $H$-good for general graphs $H$.
This study can be traced back to Bondy and Erd\H{o}s~\cite{BE73} who proved that $C_n$ is $K_k$-good whenever $n \geq k^2-2$, which led Erd\H{o}s, Faudree, Rousseau and Schelp~\cite{EFRS78} to conjecture that $C_n$ is $K_k$-good whenever $n \geq k \geq 3$.
Keevash, Long and Skokan~\cite{KLS21} recently proved a strengthening of this conjecture for large $k$, showing that $n\ge C\frac{\log k}{\log\log k}$ suffices for some constant $C \geq 1$.
For smaller $k$, Nikiforov~\cite{N09} proved $n \geq 4k+2$ is sufficient and several authors have proved the conjecture for certain small values of $k$ (see \cite{CCZ08} and its references).

For Ramsey numbers of cycles versus general graphs $H$, Burr~\cite{B81}, in 1981,
proved a satisfying result that $C_n$ is $H$-good when $n$ is sufficiently large as a function of $\abs{H}$. It remains an intriguing open question to determine the threshold of $n$ below which the Ramsey number $R(C_n,H)$ behaves differently from the natural construction of Burr yielding~\eqref{lem:burr}. In particular, Allen, Brightwell and Skokan \cite{ABS} conjectured the following explicit bound.

\begin{conj}\label{conj:abs}For any graph $H$ and $n\geq\chi(H)\abs{H}$, the cycle $C_n$ is $H$-good, \ie $R(C_n, H) = (\chi(H)-1)(n-1)+\sigma(H)$.
\end{conj}

Note that we may assume $H$ to be a complete multipartite graph, as every $H$ is a subgraph of some complete $\chi(H)$-partite graph $H'$ with $\sigma(H')=\sigma(H)$, and clearly $R(C_n, H)\leq R(C_n,H')$.

Towards Conjecture~\ref{conj:abs}, Pokrovskiy and Sudakov \cite{PS20} very recently proved an important case when the graph $H$ has polynomially small chromatic number: $\abs{H}\ge \chi(H)^{23}$. More precisely, they showed that for $n\geq 10^{60} m_k$ and $m_1 \leq m_2 \leq\cdots\leq m_k$ satisfying $m_i\geq i^{22}$ 
for each $i$, $R(C_n ,K_{m_1,\ldots,m_k} ) = (n - 1)(k - 1) + m_1$. Note, however, that it is well-known from random graph theory (e.g.~\cite{Bol88}) that for almost all graphs $H$, the chromatic number is much larger: $\chi(H)=\Theta\big(\frac{\abs{H}}{\log\abs{H}}\big)$.

Our main result below is an almost optimal quantitative version of Burr's result~\cite{B81}. In particular, this theorem holds for all graphs $H$. Moreover, as $x> C\log^4 x$ for all large $x$, it proves a strengthening of Conjecture~\ref{conj:abs} for all graphs $H$ with large (constant) chromatic number.

\begin{theorem}\label{thm:main}
There exists a constant $C > 0$ such that for any graph $H$ and any $n\geq C\abs{H}\log^4\chi(H)$ we have that $C_n$ is $H$-good, \ie $R(C_n, H) = (\chi(H)-1)(n-1)+\sigma(H)$.
\end{theorem}

The bound on $n$ in Theorem~\ref{thm:main} is best possible up to the logarithmic factor $\log^4\chi(H)$. The following construction shows that the cycle length $n$ has to be at least $(1-o(1))\abs{H}$ in order to be $H$-good. 

\smallskip

\noindent\emph{Lower bound construction.}
Fix arbitrary $m,k\in\mathbb{N}$ with $k\ge 2$ and $0<\varepsilon<\frac{1}{4}$. Set $n=(1-\varepsilon)mk$. Consider the complete $k$-partite graph $H$ on partite set $V_1,\ldots,V_k$ with $\abs{V_1}=(1-\varepsilon)mk$ and $\abs{V_2}=\cdots=\abs{V_k}=\frac{\varepsilon mk}{k-1}$; so $\abs{H}=mk$ and $\sigma(H)=\frac{\varepsilon mk}{k-1}$. Let $G$ be the $N$-vertex graph, with $N=k(n-1)$, consisting of $k$ vertex disjoint cliques $K_{n-1}$. It is easy to see that $G$ is $C_n$-free with $H$-free complement. Therefore, for these choices of $C_n$ and $H$, we have 
\[R(C_n,H)>k(n-1)>(\chi(H)-1)(n-1)+\sigma(H).\]

Our proof takes a similar approach to the work of Pokrovskiy and Sudakov \cite{PS20} and Keevash, Long and Skokan~\cite{KLS21}. The bulk of the work is to prove a stability type result showing that $C_n$-free graphs $G$ with $H$-free complement whose order is around the lower bound in~\eqref{lem:burr} must be structurally close to Burr's construction. The two novel ingredients at the heart of our proof are (i) certain sublinear expansion properties and (ii) an `adjuster' structure, both of which are inspired by recent work of Liu and Montgomery \cite{LM20+} on cycle embeddings in sublinear expanders. 
The theory of sublinear expanders has played a pivotal role in the resolutions of several old conjectures, see e.g.~\cite{FKKL,H-K-L,KLShS17,LM1,LM20+}.

The logarithmic factor in our bound is an artifact of the use of sublinear expansion. Considering the above construction, it is not inconceivable that already when $n\ge (1+o(1))\abs{H}$, the cycle $C_n$ is $H$-good. It would be interesting to at least get rid of the logarithmic factor and obtain a bound linear in $\abs{H}$.

\smallskip

\noindent\textbf{Organisation.} The rest of the paper is organised as follows. We give an outline in Section~\ref{sec:outline}. Preliminaries are given in Section~\ref{sec:prelim}. The two main ingredients, the sublinear expansion and the adjuster structure (and related lemmas) are given in Sections~\ref{sec:expansion} and~\ref{sec:gadgets}, respectively. The stability result and the proof of the main theorem are in Section~\ref{sec:stability}.

\section{Outline}\label{sec:outline}

Suppose that $G$ has order at least $(\chi(H)-1)(n-1)+\sigma(H)$, but $\overline G$ is $H$-free. With the given condition on $n$, it is not too hard to find a cycle of length \emph{at least $n$} in $G$; the difficulty lies in obtaining a cycle of the \emph{precise} desired length $n$. A natural approach to deal with this is to create a sufficiently large structure which has inbuilt flexibility about the length of cycles it can produce. 
Pokrovskiy and Sudakov \cite{PS20} use expansion properties to create some complex gadgets, which can be joined together to produce structures capable of producing paths of a wide range of lengths, up to almost all of their total size. However, the connectivity properties needed to link up these structures require $n$ to be very large compared to $\chi(H)$, and in particular only produce a good bound on $n$ if $\abs{H}$ is at least a large power of $\chi(H)$.

To deal with the missing regime where $\abs{H}$ is bounded by a polynomial in $\chi(H)$, we use an orthogonal approach. Instead of the complex gadgets featured in \cite{PS20}, we borrow ideas from recent work of Liu and Montgomery \cite{LM20+} to use certain sublinear expansion properties (Definition \ref{def:expand}) to find simpler gadgets which we call adjusters (Definition~\ref{defn:adjuster}). However, these adjusters are less flexible in terms of how much we can vary the length of the final cycle, relative to its total length. We circumvent this potential difficulty by creating a cycle which has enough adjusters to permit slightly more than $\abs{H}$ different lengths, and also has a long adjuster-free section. The complement being $H$-free means that we may shorten this section until the cycle is close enough to the desired length, and use the flexibility from the adjusters to deliver the final blow.

One of the main technical difficulties in our proof is showing that a suitable expanding subgraph may be found in a graph with $H$-free complement under our definition of sublinear expansion. Additionally, we need to work directly with the specific graph $H$, rather than passing to a balanced multipartite graph, and the possibility that the parts of $H$ are very unbalanced creates additional difficulties when one tries to use induction and it is necessary to reduce the number of parts first. We prove this, together with some useful consequences of expansion, in Section \ref{sec:expansion}.

We then leverage these expansion properties to find the adjusters that we need in Section \ref{sec:gadgets}, and show how these may be combined together to give a suitable long cycle in a sufficiently well-connected subgraph. As a consequence of sublinear expansion and the fact that we do not need as much flexibility in length adjustment, we require only a very weak connectivity condition.

Finally, in Section \ref{sec:stability} we establish a stability result on graphs $G$ which have close to $(\chi(H)-1)(n-1)$ vertices, no copy of $C_n$, no copy of $H$ in the complement, and which cannot be reduced to a smaller example by removing a part from $H$. That is, we show that a small number of vertices may be removed from such a graph $G$ to leave $\chi(H)-1$ reasonably well-connected subgraphs of order close to $n$, whose complements exclude a complete bipartite graph $H'$; here our low connectivity requirements will obviate the need for a lower bound on the class sizes of $H$. The Ramsey goodness of the cycle then follows quickly from this stability result by considering how $2$-connected blocks in $G$ are linked.

\section{Preliminaries}\label{sec:prelim}

Our aim is to use induction on the number of partite sets $k$. We therefore define how we order possible graphs $H$. 
For a complete multipartite graph $H$, we write $H'\sqsubset H$ if there exists a graph $H''$ such that $H\cong H'\vee H''$, where $\vee$ denotes graph join, \ie~$H$ is obtained by taking disjoint copies of $H'$ and $H''$ and then adding all edges between $V(H')$ and $V(H'')$.
Informally, $H'$ consists of the subgraph induced by a proper subcollection of the parts of $H$. We will sometimes use $H'\sqsubseteq H$ to mean ``$H'\sqsubset H$ or $H'=H$'', \ie $H'$ is induced by a (not necessarily proper) subcollection of parts of $H$.
For the main step of the induction, we will require both $H'\sqsubset H$ and $\sigma(H')=\sigma(H)$, 
but in some intermediate steps this latter condition is not needed. For a graph $H$ with $\chi(H)=k$, we write $m=m(H)=\frac{\abs{H}}{k}$ for the average part size of $H$, \ie $\abs{H}=km$.

For a graph $G$ and vertex set $A\subset V(G)$, we define the \textit{external neighbourhood} $N_G(A)$ to be the set $\{w\in V(G)\setminus A:vw\in E(G)\}$; note that this is disjoint from $A$. We omit the subscript if the graph is clear from context. The subgraph induced on $A$ will be denoted by $G[A]$. We write $G-A=G[V(G)\setminus A]$ for the subgraph obtained by removing vertices in $A$. For disjoint subsets $A,B\subset V(G)$, an $A$-$B$ path is a path between a vertex of $A$ and a vertex of $B$. For two vertices $u,v\in V(G)$, the graph distance $\operatorname{dist}_G(u,v)$ is the length of a shortest $u$-$v$ path in $G$. 

Logarithms with no specified base are always taken to the base $e$ throughout. For an integer $t$, we write $[t]$ for the set $\{1,2,\ldots,t\}$.

We will need the following result of Erd\H{o}s and Szekeres \cite{ES35}.
\begin{theorem}\label{thm:ES}Any sequence of at least $(r-1)^2+1$ integers contains a monotonic subsequence of length $r$.
\end{theorem}

We use the following results of Pokrovskiy and Sudakov \cite{PS20}. 
\begin{cor}\label{PS-cor}For $n\geq 10^{60} m_k$ and $m_1 \leq m_2 \leq\cdots\leq m_k$ satisfying $m_k\geq k^{22}$, 
we have $R(C_n ,K_{m_1,\ldots,m_k} ) = (n - 1)(k - 1) + m_1$.\end{cor}
\begin{proof}Apply the aforementioned result~\cite{PS20} for $K_{m_1,\ldots,m_k}$ with $m_i\geq i^{22}$ 
for each $i$ to $K_{m'_1,\ldots, m'_k}$ where $m'_1=m_1$ and $m'_i=m_k$ if $i>1$. 
For each $i>1$ we have $m'_i=m_k\geq k^{22}\geq i^{22}$, and clearly $m'_1\geq 1^{22}$. 
Thus if $G$ is a $C_n$-free graph on $(n - 1)(k - 1) + m_1$ vertices, then $\overline{G}$ contains a copy of $K_{m'_1,\ldots, m'_k}$, 
which in turn contains $K_{m_1,\ldots, m_k}$.\end{proof}

Our induction may reduce to the case when $H$ is a complete bipartite graph. 
\begin{cor}[{\cite[Corollary 3.8]{PS20}}]\label{PS-cor38}
Let $n, m_1, m_2$ be integers with $m_2 \geq m_1$ , $m_2 \geq 8$, and $n \geq 2\times 10^{49}m_2$. Then we have $R(C_n, K_{m_1 ,m_2}) = n + m_1 - 1$.
\end{cor}

The last one we need is an intermediate result on path embeddings.
\begin{lemma}[{\cite[Lemma 3.7]{PS20}}]\label{PS-lemma}
Let $n$ and $m$ be integers with $n \geq 2\times10^{49}m$ and $m \geq 8$. Let $G$ be a graph with $\overline G$ being $K_{m,m}$-free and $\abs{ N_G (A) \cup A} \geq n$ for every $A\subseteq V (G)$ with $\abs{A} \geq m$. Let $x$ and $y$ be two vertices
in $G$ such that there exists an $x$-$y$ path with order at least $8m$. Then there is an $x$-$y$ path of order exactly $n$ in $G$.
\end{lemma}

\subsection{Adjusters}
In order to construct cycles covering a range of lengths, we will use the following graphs called \emph{adjusters}. In the following definition, $m$ and $k$ are fixed numbers. When constructing adjusters, the relevant values of $m$ and $k$ will be clear from context.

\begin{defn}\label{defn:adjuster}For $r\geq 1$, an \emph{$r$-adjuster} consists of $r$ disjoint odd cycles $C_1,\ldots,C_r$, with $C_i$ having distinguished 
vertices $v_i,w_i$ which are almost-antipodal (\ie $\operatorname{dist}_{C_i}(v_i,w_i)=(\abs{C_i}-1)/2$), together with paths 
$P_1,\ldots,P_r$ where the endpoints of $P_i$ are $w_i$ and $v_{i+1}$ (subscripts taken modulo $r$), 
such that the paths are internally disjoint from each other and from the cycles, and $\abs{C_i} \leq 2000\log k\log(km)$. For the special case $r=0$, a $0$-adjuster is simply a cycle.

We refer to the cycles $C_i$ as the \emph{short cycles} of the adjuster. 
There are also cycles which use all the paths $P_i$ and part of each short cycle; we refer to these as the \emph{routes} of the adjuster. An $r$-adjuster contains $2^r$ routes, whose lengths are $r+1$ consecutive integers. 
We define the \emph{length} of an adjuster to be the length of its longest route. \end{defn}

\begin{figure}[ht]
\centering
\begin{tikzpicture}
\draw (0,0) arc [start angle=-90, end angle=90, radius=0.6];
\draw[red, line width=2pt, dash pattern=on 4pt off 4pt] (0,0) arc [start angle=-90, end angle=90, radius=0.6];
\draw (0,0) arc [start angle=270, end angle=90, y radius=0.6, x radius=0.5] node[midway, anchor=east] {$C_r$};
\filldraw (0,0) circle (0.05) node[anchor=north east] {$w_r$};
\filldraw (0,1.2) circle (0.05);

\draw (0,1.2) arc [start angle=180, end angle=90, radius=1];
\draw[red, line width=2pt, dash pattern=on 4pt off 4pt] (0,1.2) arc [start angle=180, end angle=90, radius=1];

\draw (1,2.2) arc [start angle=-180, end angle=0, radius=0.6];
\draw[red, line width=2pt, dash pattern=on 4pt off 4pt] (1,2.2) arc [start angle=-180, end angle=0, radius=0.6];
\draw (1,2.2) arc [start angle=180, end angle=0, x radius=0.6, y radius=0.5];
\filldraw (1,2.2) circle (0.05);
\filldraw (2.2,2.2) circle (0.05);

\draw (2.2,2.2) -- (3.4,2.2);
\draw[red, line width=2pt, dash pattern=on 4pt off 4pt] (2.2,2.2) -- (3.4,2.2);

\draw (3.4,2.2) arc [start angle=-180, end angle=0, radius=0.6];
\draw[red, line width=2pt, dash pattern=on 4pt off 4pt] (3.4,2.2) arc [start angle=-180, end angle=0, radius=0.6];
\draw (3.4,2.2) arc [start angle=180, end angle=0, x radius=0.6, y radius=0.5];
\filldraw (3.4,2.2) circle (0.05);
\filldraw (4.6,2.2) circle (0.05);

\draw (4.6,2.2) -- (5.8,2.2);
\draw[red, line width=2pt, dash pattern=on 4pt off 4pt] (4.6,2.2) -- (5.8,2.2);

\draw (5.8,2.2) arc [start angle=-180, end angle=0, radius=0.6];
\draw (5.8,2.2) arc [start angle=180, end angle=0, x radius=0.6, y radius=0.5];
\draw[red, line width=2pt, dash pattern=on 4pt off 4pt] (5.8,2.2) arc [start angle=180, end angle=0, x radius=0.6, y radius=0.5];
\filldraw (5.8,2.2) circle (0.05);
\filldraw (7,2.2) circle (0.05);

\draw (7,2.2) -- (8.2,2.2);
\draw[red, line width=2pt, dash pattern=on 4pt off 4pt] (7,2.2) -- (8.2,2.2);

\draw (8.2,2.2) arc [start angle=-180, end angle=0, radius=0.6];
\draw (8.2,2.2) arc [start angle=180, end angle=0, x radius=0.6, y radius=0.5];
\draw[red, line width=2pt, dash pattern=on 4pt off 4pt] (8.2,2.2) arc [start angle=180, end angle=0, x radius=0.6, y radius=0.5];
\filldraw (8.2,2.2) circle (0.05);
\filldraw (9.4,2.2) circle (0.05);

\draw (9.4,2.2) arc [start angle=90, end angle=0, radius=1];
\draw[red, line width=2pt, dash pattern=on 4pt off 4pt] (9.4,2.2) arc [start angle=90, end angle=0, radius=1];

\draw (10.4,0) arc [start angle=270, end angle=90, radius=0.6];
\draw[red, line width=2pt, dash pattern=on 4pt off 4pt] (10.4,0) arc [start angle=270, end angle=90, radius=0.6];
\draw (10.4,0) arc [start angle=-90, end angle=90, y radius=0.6, x radius=0.5];
\filldraw (10.4,0) circle (0.05);
\filldraw (10.4,1.2) circle (0.05);

\draw (0,0) arc [start angle=-180, end angle=-90, radius=1] node[midway, anchor=north east] {$P_r$};
\draw[red, line width=2pt, dash pattern=on 4pt off 4pt] (0,0) arc [start angle=-180, end angle=-90, radius=1];

\draw (2.2,-1) -- (3.4,-1) node [midway, anchor=north] {$P_1$};
\draw[red, line width=2pt, dash pattern=on 4pt off 4pt] (2.2,-1) -- (3.4,-1);

\draw (1,-1) arc [start angle=-180, end angle=0, x radius=0.6, y radius=0.5] node [midway, anchor=north] {$C_1$};
\draw (1,-1) arc [start angle=180, end angle=0, radius=0.6];
\draw[red, line width=2pt, dash pattern=on 4pt off 4pt] (1,-1) arc [start angle=180, end angle=0, radius=0.6];
\filldraw (1,-1) circle (0.05) node [anchor=south east] {$v_1$};
\filldraw (2.2,-1) circle (0.05) node [anchor=south west] {$w_1$};

\draw (4.6,-1) -- (5.8,-1) node [midway, anchor=north] {$P_2$};
\draw[red, line width=2pt, dash pattern=on 4pt off 4pt] (4.6,-1) -- (5.8,-1);

\draw (3.4,-1) arc [start angle=180, end angle=0, radius=0.6];
\draw[red, line width=2pt, dash pattern=on 4pt off 4pt] (3.4,-1) arc [start angle=180, end angle=0, radius=0.6];
\draw (3.4,-1) arc [start angle=-180, end angle=0, x radius=0.6, y radius=0.5] node [midway, anchor=north] {$C_2$};
\filldraw (3.4,-1) circle (0.05) node [anchor=south east] {$v_2$};
\filldraw (4.6,-1) circle (0.05) node [anchor=south west] {$w_2$};

\draw (5.8,-1) arc [start angle=-180, end angle=0, x radius=0.6, y radius=0.5];
\draw[red, line width=2pt, dash pattern=on 4pt off 4pt] (5.8,-1) arc [start angle=-180, end angle=0, x radius=0.6, y radius=0.5];
\draw (5.8,-1) arc [start angle=180, end angle=0, radius=0.6];
\filldraw (5.8,-1) circle (0.05);
\filldraw (7,-1) circle (0.05);

\draw (7,-1) -- (8.2,-1);
\draw[red, line width=2pt, dash pattern=on 4pt off 4pt] (7,-1) -- (8.2,-1);

\draw (8.2,-1) arc [start angle=180, end angle=0, radius=0.6];
\draw[red, line width=2pt, dash pattern=on 4pt off 4pt] (8.2,-1) arc [start angle=180, end angle=0, radius=0.6];
\draw (8.2,-1) arc [start angle=-180, end angle=0, x radius=0.6, y radius=0.5];
\filldraw (8.2,-1) circle (0.05);
\filldraw (9.4,-1) circle (0.05);

\draw (9.4,-1) arc [start angle=-90, end angle=0, radius=1];
\draw[red, line width=2pt, dash pattern=on 4pt off 4pt] (9.4,-1) arc [start angle=-90, end angle=0, radius=1];
\end{tikzpicture}
\label{fig:gadget-ex}\caption{An $r$-adjuster. An example of a route in this $r$-adjuster is given by the red dashed line.}
\end{figure}
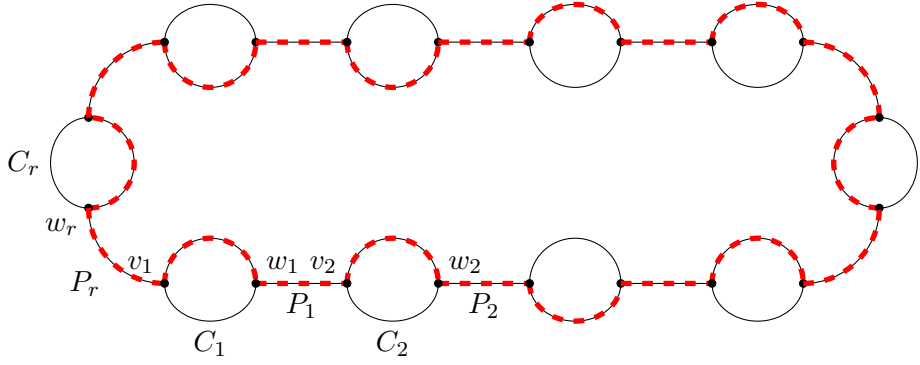

\section{Sublinear expansions}\label{sec:expansion}

Given a large graph whose complement is $H$-free we will pass to a subgraph which is not too large 
and has sublinear expansion properties, in which we can find an $r$-adjuster for some suitable $r$. 
After removing this $r$-adjuster we repeat the process with the remaining graph. 
We then join many adjusters together to create a very large adjuster with greater flexibility in the length of a route, 
and where this length can exceed $n$. We will shrink the structure, if necessary, to ensure that 
the maximum length of a cycle is not much more than $n$, without impairing this flexibility. 
We then show this structure contains $C_n$.

In this section we define the expansion properties we need, and show that any sufficiently large graph whose complement is $F$-free, for some graph $F$, contains an expanding subgraph of suitable size.
\begin{defn}\label{def:expand}
  A graph $G$ \emph{$(\Delta,\beta,d,k)$-expands into a set $W \subseteq V(G)$} if the following holds.
  \begin{itemize}
  \item $\abs{N_G(S) \cap W} \geq \Delta \abs{S}$, for every $S\subseteq V(G)$ with $\abs{S} \leq \beta d$.
  \item $\abs{N_G(S)} \geq \frac{\abs{S}}{10\log k}$, for every $S\subseteq V(G)$
    with $\beta d \leq \abs{S} \leq \abs{G}/2$.
  \end{itemize}
\end{defn}

Note that whenever we say that $G$ $(\Delta, \beta, d, k)$-expands, and do not specify the set $G$ expands into, we mean that $G$ $(\Delta, \beta, d, k)$-expands into $W = V(G)$.

\begin{lemma}\label{lem:expansion} Fix a complete $k$-partite graph $H$ of order $mk$, where $k\geq 2$, then for all $\beta, M, \Delta \geq 1$ with $M \geq 60 \beta \geq 240 \Delta$, $M \geq 10\beta \Delta$, $M \geq 4k$ and $\beta \geq 10\log k$, the following holds. Let $G$ be a graph with $\overline{G}$ being $H$-free and $\abs{G} \geq Mmk\log k$. Then there exists
	a subgraph $H'\sqsubseteq H$ induced by some collection of at least two parts and an induced subgraph $F \subseteq G$ such that the following holds:
  \begin{itemize}
	\item $\overline{F}$ is $H'$-free;
	\item $M\abs{H'}\log \chi(H') - m(H') \leq \abs{F} \leq M\abs{H'}\log \chi(H')$.
	\item $F$ is $(\Delta,\beta,m(H'),\chi(H'))$-expanding.
\end{itemize}
\end{lemma}
\begin{proof}
We proceed by induction on $k$, considering two cases. The base case $k=2$ will be covered by the inductive argument, since we argue that we may either (i) construct a suitable subgraph directly for $H'=H$, (ii) reduce to some smaller $H'\sqsubset H$ with at least two parts, or (iii) obtain a contradiction; in the base case only the first and third possibilities arise. Note that, by removing vertices if necessary, we may assume $\abs{G}=Mmk \log k$ for $k\geq 2$. We write $m_1\leq\cdots\leq m_k$ for the orders of the vertex classes of $H$.

The first case is if there exists a set $S\subset V(G)$ such
that $m \leq \abs{S} \leq \abs{G}/2$ and $\abs{N_G(S)} \leq \abs{S}/(10\log k) +
m\Delta\beta$. Let $T = V(G)\setminus (N_G(S) \cup S)$. 
Since $M \geq 10\Delta\beta $, we have that
\begin{align*}  
    \abs{T}\geq \abs{G}-\abs{N_G(S)\cup S}
                                &   \geq \left(1-1/(10\log k)\right)\abs{G}/2 - m\Delta\beta \\
                                &   = m\left(\tfrac{Mk}{2}(\log k-1/10) - \Delta\beta\right) \\
                                &   \geq km.
\end{align*} We distinguish two subcases depending on whether $\abs{S}\geq mk$ or $\abs{S}<mk$.

First, suppose $\abs{S}\geq mk$. We choose $t \in \bN$ as large as possible such that $\abs{T}\geq M\abs{H_t}\log t$, where $H_t$ is the graph induced by the $t$ smallest parts of $H$.
We then choose $s \in \bN$ as large as possible so that $\abs{S}\geq M\abs{H_{s,t}}\log s$ where $H_{s,t}$ is the graph induced by the $s$ smallest parts of $H$ not contained in $H_t$; observe that $s,t\geq 1$ and $s,t\leq k-1$. 

We will quickly be able to resolve our first subcase after proving the following claim.

\begin{claim}\label{claim:s+tgeqk} 
$s + t = k$.
\end{claim} 

\begin{poc}
The claim is trivial for $k=2$, since $s,t\geq 1$. For all $k\geq 3$ we have
\[\abs{T} \geq \abs{G}/2-\frac{\abs{G}}{20\log k} - \frac{\abs{G}}{10k\log k} \geq M\abs{H_{\ceil{k/2}}}\log(\ceil{k/2}),\] which implies $t\geq\ceil{k/2}$; in particular, this proves the claim when $k=3$.

Assume for a contradiction that $s + t < k$, where $k\geq 4$ is fixed; as $t \geq \ceil{k/2}$, we have $s\leq(k-1)/2$. Maximalities of $s$ and $t$ ensure that 

\begin{equation}\label{eq:lem2.4begin1}
    M\abs{H_{s+1, t}}\log(s+1) > \abs{S}
\end{equation} and
\begin{equation}\label{eq:lem2.4begin2}
    M\abs{H_{t+1}}\log(t+1) > \abs{T}.
\end{equation} Observe that $H_{t+1}$ and $H_{s+1, t}$ both contain the $(t+1)$th smallest part of $H$. 

Also, $\abs{S}+\abs{T} = \abs{G} - \abs{N_G(S)} > Mmk\log k - \frac{\abs{S}}{10\log k} - m\Delta\beta$, hence by \eqref{eq:lem2.4begin1}, \eqref{eq:lem2.4begin2} and the fact that $M\geq 10\beta\Delta$ we have
\begin{equation}\label{eq:lem2.4main}
    \abs{H_{t+1}}\log(t+1) + \abs{H_{s+1, t}}\log(s + 1)c(k) \geq  mk\log k - \frac{m}{10},
\end{equation}
where $c(k)=\left(1 + \frac{1}{10\log k}\right)$.
We want to show that \eqref{eq:lem2.4main} is in fact false, which provides the contradiction we need to prove Claim~\ref{claim:s+tgeqk}. To this end, we may assume $s+t = k-1$. It follows that 
\begin{align*}\abs{H_{t+1}} + \abs{H_{s+1, t}}=mk+m_{t+1}
\leq mk+\frac{\abs{H_{s+1, t}}}{s+1};\end{align*}
recall that $m_{t+1}$ is the size of the $(t+1)$th smallest part of $H$. Consequently the LHS of \eqref{eq:lem2.4main} is at most
\begin{equation}\label{LHS-ub}
    c(k)\log(s+1)\abs{H_{s+1, t}} + \log(k-s)\left(mk - \abs{H_{s+1, t}}\frac{s}{s+1}\right).
\end{equation}
Note that $m(s+1)\leq \abs{H_{s+1, t}}\leq mk$. For fixed $s$, \eqref{LHS-ub} is linear in $|H_{s+1, t}|$, and consequently within this range it is maximised either at $\abs{H_{s+1, t}}=m(s+1)$ or at $\abs{H_{s+1, t}}=mk$.

We first consider the case $\abs{H_{s+1, t}}=m(s+1)$, when \eqref{LHS-ub} becomes \[m(c(k)(s+1)\log(s+1)+(k-s)\log(k-s)).\] This is a convex function of $s$, and so for $1\leq s\leq (k-1)/2$ is maximised when $s=1$ or when $s=(k-1)/2$. When $s=1$ we have
\begin{equation}\label{bal-small-s}
2c(k)\log 2+(k-1)\log(k-1)<k\log k-0.1\end{equation}
for all $k\geq 4$. When $s=(k-1)/2$, we likewise have
\begin{equation}\label{bal-large-s}
(1+c(k))\frac{k+1}{2}\log\left(\frac{k+1}{2}\right)<k\log k-0.1\end{equation}
for all $k\geq 4$.

Finally, we consider the case $\abs{H_{s+1, t}}=mk$, when \eqref{LHS-ub} becomes \[mk(\log(s+1)c(k)+\log(k-s)/(s+1)).\]
Suppose $2\leq s\leq \log(k-s)$. Then $\frac{s}{s+1}\log(k-s)\geq\frac{s^2}{s+1}\geq c(k)\log(s+1)+0.1$, and so \eqref{LHS-ub} is decreasing in $\abs{H_{s+1, t}}$, and maximised when $\abs{H_{s+1, t}}=m(s+1)$. Similarly if $s=1$ and $k\geq 7$, \eqref{LHS-ub} is decreasing. For $s=1$ and $4\leq k\leq 6$, by direct calculation \eqref{LHS-ub} contradicts \eqref{eq:lem2.4main}. Thus we may assume that $s\geq\log(k-s)$. Since $s\leq(k-1)/2$ we have $s+1\leq k-s$. Thus
\[\frac{\mathrm{d}}{\mathrm{d}s}\left(\log(s+1)c(k)+\frac{\log(k-s)}{s+1}\right)=\frac{c(k)(s+1)-\frac{s+1}{k-s}-\log(k-s)}{(s+1)^2}>0\]
and so $\log(s+1)c(k)+\log(k-s)/(s+1)$ is increasing in $s$ for $\log(k-s)\leq s\leq (k-1)/2$. Hence within this range $mk(\log(s+1)c(k)+\log(k-s)/(s+1))$ is maximised when $s=(k-1)/2$, and we obtain
\begin{equation}\label{unbal-large-s}k\left(c(k)+\frac{2}{k+1}\right)\log\left(\frac{k+1}{2}\right)<k\log k - 0.1\end{equation}
for all $k\geq 4$. Then \eqref{bal-small-s}, \eqref{bal-large-s} and \eqref{unbal-large-s} together show that \eqref{LHS-ub} is less than $mk\log k - \frac{m}{10}$, contradicting \eqref{eq:lem2.4main}.
\end{poc}

Hence $s+t= k$. Now either $\overline{G[T]}$ is $H_t$-free or $\overline{G[S]}$ is $H_{s,t}$-free, since otherwise disjointness of $T$ and $N_G(S)\cup S$ ensures we have a copy of $H$ in $\overline{G[S\cup T]}$. However, since $\abs{S},\abs{T}\geq mk$, $\overline{G[T]}$ is not $H_1$-free and $\overline{G[S]}$ is not $H_{1,k-1}$-free. Thus in the former case we have $2\leq t\leq k-1$ so we may apply the induction hypothesis to $G[T]$ with $k$ replaced by $t$ and $H$ by $H_t$, and in the latter $2\leq s\leq k-1$ so we may apply the induction hypothesis to $G[S]$ with $k$ replaced by $s$ and $H$ by $H_{s,t}$. Here, such applications of the induction hypothesis are possible by the definition of $s$ and $t$.

Secondly, suppose $m\leq \abs{S}<mk$. 
Note that, since $M \geq 10\Delta\beta$ and $M \geq 4k$, in this case \[\abs{T}\geq Mmk\log k-2mk-m\Delta\beta \geq M(mk-1)\log (k-1).\] Indeed, $Mmk\log \frac{k}{k-1} \geq Mm \geq 2mk + m\Delta\beta$ where we have used $\log(1+x) \geq \frac{x}{1+x}$ with $x = \frac{1}{k-1}$. Thus $\abs{T} \geq M\abs{H_{k-1, 1}}\log (k-1)$ and $\abs{S} \geq m \geq \abs{H_1}$. Since $H_1$ is an independent set of size at most $m$, $\overline{G[S]}$ contains $H_1$. 
Thus $\overline{G[T]}$ must be $H_{k-1, 1}$-free, as otherwise $\overline{G[S\cup T]}$ contains $H$, as before. Since $\abs{T}\geq mk$, it is not $H_{1,1}$-free and so $k-1\geq 2$.
We may therefore apply the induction hypothesis to $G[T]$ with $k-1$ replacing $k$ and $H_{k-1,1}$ replacing $H$.  

The final case is where we have
$\abs{N_G(S)\cup S} \geq \left(1+1/(10\log{k})\right)\abs{S} +  m\Delta\beta$ for every set $S \subseteq V(G)$
with $m \leq \abs{S} \leq \abs{G}/2$. Consider the largest set $X \subseteq V(G)$ with $\abs{X}
\leq 2m$ such that $\abs{N_G(X)} \leq \Delta \abs{X}$.
Since $\beta \geq 4\Delta$, we have that $\abs{N_G(X) \cup X} \leq (\Delta + 1)\abs{X} \leq m\Delta\beta $. 
Thus, we must have $\abs{X} < m$. Now consider $F = G- X$. If there is some $Y\subset V(F)$ with $\abs{Y}\leq m$ and $N_F(Y) < \Delta\abs{Y}$, then
\[
    \abs{N_G(X \cup Y) \cup (X \cup Y)} 
                     \leq \abs{N_G(X) \cup X} + \abs{N_F(Y) \cup Y} < (\Delta + 1)\abs{X\cup Y},
\]
contradicting our choice of $X$. 

For $\abs{Y}\geq m$, we have 
\[\abs{N_F(Y) \cup Y} \geq \abs{N_G(Y) \cup Y} - \abs{X}\geq \left(1 + \frac{1}{10\log{k}}\right)\abs{Y} + m\Delta\beta - m.\] 
If $\abs{Y}\leq \beta m$ then utilising $\beta \geq 10\log k$ we have that $\abs{N_F(Y) \cup Y} \geq\abs{Y}+\Delta\abs{Y}$, whereas if $\beta m \leq \abs{Y} \leq \abs{F}/2$ then $\abs{N_F(Y) \cup Y}\geq \left(1 + 1/10\log{k}\right)\abs{Y}$. Thus $F$ fulfils the requirements of Lemma~\ref{lem:expansion} for $H'=H$.
\end{proof}

These expansion properties give us three abilities we will need to construct adjusters. 
The first is the ability to link together large sets, while avoiding a smaller set, via a path which is not too long. 
The main advantage of our stronger definition of expansion (compared to that in \cite{PS20}) is that the length of path required is much shorter, 
being only of order $\log k\log \abs{G}$, and the control we have over the size of the expanding subgraph means 
that we can ensure $\log\abs{G}=O(\log k)$.

\begin{lemma}\label{lem:short-path}
Suppose that $G$ $(\Delta,\beta,d,k)$-expands into a set $W\subseteq V(G)$, for $\Delta\geq2$ and $k \geq 2$. 
If $A,B,C\subseteq V(G)$ are disjoint sets with $A,B\neq\varnothing$ such that 
$\abs{C\cap W} \leq (\Delta-2) \min\{\abs{A},\abs{B}\}$ and
$\abs{C}\leq \beta d/(20\log k)$, then there is an $A$-$B$ path $P$ in $G$ avoiding $C$ and with $\abs{P}\leq44\log k\log\abs{G}$.
\end{lemma}
\begin{proof}
Set $A_0=A$ and $A_{i+1}=(N(A_i)\cup A_i)\setminus C$ for each $i$. Since we have that $\abs{A_{i+1}}\geq 2\abs{A_i}$ for $\abs{A_i}<\beta d$, 
and thus $\abs{A_a}\geq \beta d$ where $a=\log_2(\beta d)$. For $i\geq a$ we have 
$\abs{N(A_i)}\geq \frac{1}{10\log k}\abs{A_i}\geq 2\abs{C}$, and so $\abs{A_{i+1}}\geq (1+1/20\log k)\abs{A_i}$. 
Consequently $\abs{A_b}\geq\abs{G}/2$, where $b=a+\log_{1+1/20\log k}(\abs{G}/2\beta d)$. Applying 
$\log(1+x)=-\log(1-\frac{x}{1+x})\geq (1+1/x)^{-1}$ for $x=1/20\log k$, and noting that $1+20\log k\leq 22\log k$, we get 
\[b\leq a+22\log k\log \abs{G}-22\log k\log(2\beta d)\leq22\log k\log \abs{G}.\]
Since the same argument applies to $B$, we have that $A_b\cap B_b\neq\varnothing$, and so there is a path of length at most $2b\leq 44\log k \log\abs{G}$ avoiding $C$, as desired.
\end{proof}
The next ingredient follows from Lemma~\ref{lem:short-path}, and allows us to find a short odd cycle 
while retaining expansion into the rest of the graph.
\begin{lemma}\label{lem:cycle}Suppose $G$ is not bipartite and $(\Delta,\beta,d,k)$-expands into $W\subseteq V(G)$ 
for $\Delta\geq 4$, $k \geq 2$ and $\beta\geq 8\Delta$. Then $G$ contains an odd cycle $C$ with length at most $88\log k\log\abs{G}$, 
such that $G$ also $(\Delta-3,\beta,d,k)$-expands into $W\setminus V(C)$.\end{lemma}
\begin{proof}Since $G$ is not bipartite, it contains an odd cycle. Let $C$ be a shortest odd cycle; 
it follows that if $x,y\in V(C)$ then $\operatorname{dist}_G(x,y)=\operatorname{dist}_C(x,y)$. Indeed, if not, choosing $x,y$ to violate this condition with minimal distance, 
the shortest path between $x$ and $y$ is internally disjoint from the cycle, and adding this path to whichever 
path around the cycle has the appropriate parity gives a shorter odd cycle. Consequently $\abs{C}$ is at most twice the diameter of $G$. 
Using Lemma~\ref{lem:short-path} with $A_{\ref*{lem:short-path}},B_{\ref*{lem:short-path}}$ being any singletons and $C_{\ref*{lem:short-path}}=\varnothing$ implies the desired bound 
$\abs{C}\leq 88\log k\log\abs{G}$.

Now, if $C$ is a triangle then any vertex $v$ can have at most 3 neighbours on $C$ (trivially). Otherwise $v$ can have at most two neighbours on $C$; indeed, if $v$ has at least 3 neighbours on $C$ then either $v$ has two neighbours $x,y$ 
with $\operatorname{dist}_C(x,y)\geq 3>\operatorname{dist}_G(x,y)$, contradicting the earlier observation, or $v$ has two neighbours $x,y$ which are adjacent,
giving a shorter odd cycle $vxy$. Thus for any set $S$ we have $\abs{N_G(S)\cap (W\setminus V(C))}\geq \abs{N_G(S)\cap W}-3\abs{S}$, and, 
since $G$ $(\Delta,\beta,d,k)$-expands into $W$, $G$ also $(\Delta-3,\beta,d,k)$-expands into $W\setminus V(C)$.
\end{proof}
Finally, we show that we can construct two large, well-connected disjoint sets centred around any two given vertices.
\begin{lemma}\label{lem:wings}Suppose $G$ is a graph which $(\Delta,\beta,d,k)$-expands into $W\subseteq V(G)$, 
and let $x,y$ be distinct vertices. Then there exist disjoint sets $A,B\subset W\cup\{x,y\}$ of size $\beta d/2$ 
such that $x\in A$, and every vertex of $A$ is connected to $x$ by a path in $A$ of length at most $\log_\Delta(\beta d/2)$, 
and similarly for $B$ and $y$.\end{lemma}
\begin{proof}We iteratively build up sets $A_i,B_i$ with $\abs{A_i}=\abs{B_i}=\Delta^i$, such that any vertex in $A_i$ 
is connected to $x$ by a path in $A_i$ of length at most $i$, until $i$ is large enough that $\abs{A_i}\geq\beta d/2$. 
We start from $A_0=\{x\},B_0=\{y\}$. Given disjoint sets $A_i,B_i$ with $\abs{A_i}=\abs{B_i}<\beta d/2$, 
we aim to choose disjoint sets $A'\subset N_G(A_i)\cap W$ and $B'\subset N_G(B_i)\cap W$ of size at least $\Delta\abs{A_i}$. 
Then we can pick $A_{i+1}\subseteq A_i\cup (A'\setminus B_i)$ and $B_{i+1} \subseteq B_i \cup (B'\setminus A_i)$ with the required disjointness, connectivity and size properties.

First, take $A''=(N(A_i) \cap W)\setminus N(B_i)$ and $B''=(N(B_i) \cap W)\setminus N(A_i)$. 
If either of these, say $A''$ has size at least $\Delta\abs{A_i}$, then we may choose 
$A'\subseteq A''$ and $B'\subseteq N(B_i) \cap W$ both of size $\Delta\abs{A_i}$ (using the expansion property for the latter); 
by definition of $A''$ these sets are disjoint.

Consequently we may assume $\abs{A''},\abs{B''}<\Delta\abs{A_i}$. Since $\abs{A_i\cup B_i}<\beta d$, 
the expansion property gives $\abs{N(A_i)\cap N(B_i) \cap W}\geq 2\Delta\abs{A_i}-\abs{A''}-\abs{B''}$, 
and so we can add disjoint subsets of $\abs{N(A_i)\cap N(B_i) \cap W}$ to $A''$ and $B''$ to form sets $A',B'$ as required.
\end{proof}

\section{Constructing adjusters}\label{sec:gadgets}
In this section we prove the following result for some suitable constant $C$.

\begin{lemma}\label{lem:sn} 
There exists a constant $C>0$ satisfying the following.
Fix integers $m,k$ with $C\leq m\leq k^{22}$ and $n\geq 8\times 10^{18}mk\log^4k$. Let $H$ be a complete $k$-partite graph on $mk$ vertices, \vspace{0.5mm}
$H' \sqsubseteq H$ be $(s+1)$-partite for some $1\leq s\leq k-1$ and $m_2$ be the size of the second smallest part of $H$.
Let $G$ be a graph of order at least $sn-n/10$ with $\overline{G}$ being $H'$-free, 
and such that any two sets of size at least $m_2+4.1\times 10^{18}\log^4 k$ have at least $4.1\times 10^{18}\log^4 k$ 
disjoint paths between them. 
\begin{itemize}\item If $s\geq 2$, then $G$ contains a cycle of length exactly $n$.
\item If $s=1$ then $G$ contains an $r$-adjuster for some $r\geq 9mk/8$ having length between $0.6n$ and $0.7n$.
\end{itemize}\end{lemma}

We first show how the expansion properties established in Section \ref{sec:expansion} can be used to construct adjusters.
\begin{lemma}\label{lem:find-gadget}Fix $k \geq 2$, $10^6 \leq m\leq k^{22}$ and $n\geq 8\times10^{5}mk\log^2 k$. 
Let $H$ be a complete $(s+1)$-partite graph of order $m'(s+1)$, where $m' \in \mathbb{R}$, $1\leq s\leq k-1$ and $\abs{H}\leq mk$. 
Let $G$ be a graph with $\overline{G}$ being $H$-free and $\abs{G}\geq sn/10$. Then $G$ contains an $r$-adjuster 
for some $r\geq \frac{mk}{2\times 10^4\log^2 k}$, which has between $mk$ and $mk + (3\times 10^4\log^2) k$ vertices.
\end{lemma}
\begin{proof}First note that we may assume, at the cost of replacing the upper bound with $\abs{H}\leq 2mk$, 
that every part of $H$ has order at least $m'$. Indeed, we may replace each part of $H$ with order below $m'$ 
with a part of order $\lceil m'\rceil$; this adds at most $sm'$ vertices, and the new graph has average 
part size at most $2m'$. In particular, this means that if $H'\sqsubseteq H$ is $a$-partite 
of order $ab$ then $b\geq m'$ and $ab\leq 2m'(s+1)\leq 2mk$. Clearly, replacing $H$ with a supergraph preserves 
the property that $\overline{G}$ is $H$-free. 

Set 
\[M=(10^{4}k\log k) \times \frac{2m}{m'}\quad\text{ and }\quad \beta=M/(200)\geq 100.\]
Then, as $1 \leq s \leq k-1$, we get
\[\abs{G}\geq 8\times10^{4}smk\log^2 k\geq M(2m'(s+1))\log (s+1).\] Further, note that $M \geq 4(s+1)$ and $\beta \geq 10 \log k$.
We use Lemma~\ref{lem:expansion} with $(M, \beta, \Delta, m, k)_{\ref*{lem:expansion}} = (M, \beta, 10, m', s+1)$ to pass to 
a subgraph $G'$ where $\overline{G'}$ is $H'$-free for some $H'\sqsubseteq H$ and which $(10,\beta,m'',s')$-expands 
into $V(G')$, where  $\abs{H'}=s'm''$ and $\chi(H')=s'$ with $2 \leq s' \leq s+1$, such that \begin{equation}\label{eq:G'}Mm''s'\log s' - m''\leq\abs{G'}\leq Mm''s'\log s'.\end{equation}
In particular, since $m''\geq m'$, we have that $G'$ $(10,\beta,m',s')$-expands. 

As $\overline{G'}$ is $H'$-free, $G'$ has no independent set of order $\abs{H'} = s'm''$, and it follows that $\chi(G')\geq\abs{G'}/(s'm'')>2$. Note that as $m \leq k^{22}$, from \eqref{eq:G'} we have that $\abs{G'} \leq M (2m'(s+1)) \log k \leq 10^{5}mk^2 \log^2 k \leq 10^{5}k^{26}$.
Thus, by Lemma~\ref{lem:cycle} we may find an odd cycle $C_1$ of length at most $88\log k \log\abs{G'}\leq 10^4\log^2 k$, 
such that we retain $(7,\beta,m',s')$-expansion into $V(G')\setminus V(C_1)$. Choosing almost-antipodal points $v_1,w_1$ on $C_1$ (\ie $\operatorname{dist}_{C_1}(v_1,w_1) = (\abs{C_1} - 1)/2$), 
we use Lemma~\ref{lem:wings} to find sets $A_1$ and $B_1$ of size $\beta m'/2$ with 
$A_1\cap B_1=\varnothing$, $A_1\cap V(C_1)=\{v_1\}$ and $B_1\cap V(C_1)=\{w_1\}$, 
with each vertex in $A_1$ (resp. $B_1$) having a path in $A_1$ to $v_1$ (resp. in $B_1$ to $w_1$) of length at most $\log_{7}(\beta m'/2)<40\log^2 k$.

Next we set $X=A_1\cup V(C_1)\cup B_1$ and $Y=V(C_1)$. 
Note that, by \eqref{eq:G'} and that $\abs{X} \leq \beta m' + 10^4 \log^2 k$, we have
$\abs{G' - X}\geq \frac{1}{2}Mm''s'\log s'$. 
We apply Lemma~\ref{lem:expansion} with $(M, \beta, \Delta, m, k)_{\ref*{lem:expansion}} = (M/2, \beta, 10, m'', s')$  again to find a subgraph $G''$ with 
$V(G'')\subset V(G')\setminus X$ which $(10,\beta, m''',s'')$-expands and has order at most $\frac{1}{2}Mm'''s''\log s''$, 
for some $2 \leq s''\leq s'$ and $m'''\geq m'$. 
Within $G''$ we find another similar structure $A_2\cup C_2\cup B_2$, 
and since $G'$ $(10, \beta, m', s')$-expands we can use Lemma~\ref{lem:short-path}, with parameters $(\Delta, \beta, d, k, A, B, C, W)_{\ref*{lem:short-path}} = (10, \beta, m', s', A_2, B_1, V(C_1 \cup C_2), G')$ to find a path $Q_1$ between $A_2$ and $B_1$ 
that avoids $V(C_1\cup C_2)$ and is of length at most $5000\log^2 k$. We may assume $Q_1$ is disjoint from $A_1\cup B_2$ 
since otherwise we may find a shorter path between (say) $A_1$ and $A_2$, then relabel appropriately. 
We then extend $Q_1$ by paths within $A_2$ to $v_2$ and within $B_1$ to $w_1$ to obtain a $w_1$-$v_2$ path $P_1$ of length at most $10^4\log^2 k$.

Now we update $X$ to $A_1\cup V(C_1\cup P_1\cup C_2)\cup B_2$ (note that the unused parts of $A_2\cup B_1$ are released) 
and $Y$ to $V(C_1\cup P_1\cup C_2)$. As before, we choose $G'''$ with $V(G''')\subset V(G')\setminus X$ such that $G'''$ $(10,\beta, m^{(4)},s''')$-expands 
for some $2 \leq s'''\leq s'$ and $m^{(4)}\geq m'$ and continue as before, updating $X$ and $Y$.
Observe that we can continue this process as long as $\abs{G' - X} \geq  \frac{1}{2}Mm''s'\log s'$. Thus, since $\abs{G'} \geq Mm''s'\log s' - m''$, we can continue whenever $\abs{X} \leq \frac{1}{4}Mm''s'\log s'$. Since $mk + \beta m' \leq \frac{1}{4}Mm''s'\log s'$, we can continue the process until $\abs{Y}>mk$, as $\abs{A_i}, \abs{B_i} \leq \beta m'/2$ for all $i$. Moreover, at each stage of this process $\abs{Y}$ increases by at most $2\times 10^4\log^2 k$. Thus as soon as $\abs{Y}>mk$ we can stop the process and ensure $\abs{Y} \leq mk + (2\times 10^4)\log^2 k$
at this point. Since $mk + 2\times 10^4\log^2 k \leq \beta m'/20\log (s+1) \leq \beta m'/20\log s'$, we can then apply Lemma~\ref{lem:short-path} with $(\Delta, \beta, d, k, A, B, C, W)_{\ref*{lem:short-path}} = (10, \beta, m', s', A_1,  B_r, Y, G')$ to find a path of length at most $10^4\log^2 k$ connecting the two remaining large sets $A_1,B_r$ avoiding $Y$, 
which closes up the structure into an $r$-adjuster; since each cycle $C_i$ and each path $P_i$ built in this whole process has length at most $10^4\log^2 k$ 
we have $r \geq \frac{mk}{2 \times 10^4\log^2 k}$, and the adjuster contains all of $Y$, giving the required bounds.
\end{proof}

We will also need some long cycles to extend the adjusters we construct.
\begin{lemma}\label{lem:long-cycle} Fix $k \geq 2$, $10^6 \leq m\leq k^{22}$ and $n\geq 8\times10^{5}mk\log^2 k$. 
Let $H$ be a complete $(s+1)$-partite graph of order $m'(s+1)$, where $m' \in \mathbb{R}$, $1\leq s\leq k-1$ and $\abs{H}\leq mk$. 
Let $G$ be a graph with $\overline{G}$ being $H$-free and $\abs{G}\geq sn/10$. 
Then $G$ contains a cycle of length between $n/((8\times 10^5)\log^2 k)$ and $n/((8\times 10^5)\log^2 k) + 2mk$.
\end{lemma}
\begin{proof} We proceed as in Lemma~\ref{lem:find-gadget}, with the following differences. First, set
\[M=\frac{n}{10m'\log(s+1)}\geq\frac{n}{10m'\log k}\quad\text{and}\quad\beta=M/200.\] 
Secondly, at each stage rather than finding a cycle with two large sets, each of size $\beta m'/2$, attached we find a single edge with two large sets. 
This ensures that after each joining step we have a path with two large sets attached. 
Again, we continue until the total size of the path exceeds $\beta m'/40\log k$, and then join the two large sets, using Lemma~\ref{lem:short-path}, to obtain a desired long cycle. If this cycle is longer than our desired upper bound, we truncate it using that $\overline{G}$ is $H$-free (choose $mk$ vertices that are non-adjacent on the cycle).
\end{proof}

In order to complete the proof of Lemma~\ref{lem:sn}, we need to be able to join adjusters together while simultaneously retaining 
most of the flexibility of each adjuster and most of the length. This resembles Lemma~2.19 of \cite{PS20}, 
but the fact that in our case the length of an adjuster and the number of vertices are only loosely related creates extra difficulties.
\begin{lemma}\label{lem:merge-gadgets}Let $F_i$ be an $r_i$-adjuster of length $\ell_i$ for $i\in\{1,2\}$, 
with $F_1$ and $F_2$ disjoint. Suppose there are $s>16$ vertex-disjoint paths $P_1,\ldots,P_s$ between 
$F_1$ and $F_2$, each of length at most $t$. Then for some $i\neq j$ there is an $r$-adjuster of length $\ell$ 
contained in $F_1\cup F_2\cup P_i\cup P_j$, where
\[(r_1+r_2)(1-4s^{-1/2})-4\leq r\leq r_1+r_2\]
and
\[(\ell_1+\ell_2)(1-4s^{-1/2})\leq\ell\leq\ell_1+\ell_2+2t.\]
Furthermore, if $r_2=0$ then the adjuster obtained contains a section of length $\ell-\ell_1$ with no short cycles.
\end{lemma}
\begin{proof}We may assume each $P_i$ is internally disjoint from $F_1\cup F_2$, since otherwise we can replace 
it by a shorter path with this property. Write $x_i$ for the endpoint of $P_i$ in $F_1$, and $y_i$ for the endpoint in $F_2$. 

We will progressively reduce the number of paths to focus our attention on at least $s^{1/2}/4$ paths which relate 
to one another in a useful way. First, we give each path $P_i$ an ordered pair of labels from $\{+,-\}$ as follows. 
If $x_i$ lies on one of the $r_1$ short cycles of $F_1$, choose the first label to be $+$ if it is on the longer section of 
that cycle (that is, the longer section between the two vertices of that short cycle that have degree $3$ in $F_1$), and $-$ if it is on the shorter section. If $x_i$ lies on a path between cycles, choose the first label arbitrarily. 
Choose the second label similarly with respect to $y_i$. Now we may find a set of at least $s/4$ paths with identical label pairs. 
Note that this ensures there are routes $C_1$ around $F_1$ and $C_2$ around $F_2$ which cover all the $x_i$ and $y_i$ 
for $P_i$ in this set of paths, with $\abs{C_i}\in[\ell_i-r_i,\ell_i]$. In what follows we consider $C_1$ and $C_2$ to be oriented in a particular direction around $F_1$ and $F_2$, respectively. Renumbering, if necessary, we may assume these paths 
are $P_1,\ldots,P_{\ceil{s/4}}$, and that $x_1,\ldots,x_{\ceil{s/4}}$ are in order of their appearance in $C_1$. 
By the Erd\H{o}s--Szekeres Theorem (Theorem \ref{thm:ES}), there is a subset of at least $t\geq\left\lceil\sqrt{s/4}\right\rceil\geq 2$ 
paths with endvertices $y_i$ either in order of their appearance in $C_2$ or in reverse order. 
Again, by renumbering and reversing $C_2$ if necessary we may assume these are $P_1,\ldots,P_{t}$ and the $y_i$ appear in reverse order.

Let $C'_1$ be the longest route in $F_1$, and associate each vertex $x_i$ for $i\in[t]$ with a vertex $x'_i$,
chosen as follows. If $x_i$ is internal to a shorter side of a short cycle, let $x'_i$ be the vertex on the longer side 
which is the same distance from the start vertex of the cycle (consequently, $x'_i$ is one step further from the end of 
the cycle than $x_i$). If $x_i$ is on one of the paths or is internal to the longer side of a short cycle, set $x'_i=x_i$. 
Define $C'_2$ and each $y_i'$ similarly. 

Now we show that some consecutive pair of paths work. Consider making a new adjuster from $F_1$, $F_2$, $P_i$ and 
$P_{i+1}$ for $i\in[t]$ (we take subscripts modulo $t$) as follows: starting from $x_i$, traverse $P_i$ to $y_i$, 
then follow a route around $F_2$ to $y_{i+1}$, including both sections of any short cycle traversed in this route (apart from any short cycle containing $y_i$ or $y_{i+1}$), 
traverse $P_{i+1}$ to $x_{i+1}$, then follow a route around $F_1$ to $x_i$ similarly. See Figure \ref{fig:gadget-merge} for an example.

Write $S_1$ for the edges of $C'_1$ between $x'_i$ and $x'_{i+1}$ and $S_2$ for the edges of $C'_2$ between 
$y'_{i+1}$ and $y'_{i}$. The short cycles of the new adjuster are those of the original two adjusters except for any 
which intersect $S_1$ or $S_2$. The longest route through the new adjuster uses all of $C'_1$ except for at most $\abs{S_1} + 1$ edges (the $+1$ added if $x'_{i+1}\neq x_{i+1}$), and all of $C'_2$ except for at most $\abs{S_2} + 1$ edges (the $+1$ added if $y'_{i+1}\neq y_{i+1}$). It also uses all of the two paths $P_i$ and $P_{i+1}$, which contain at least one edge each, so its length is at least
$\ell_1+\ell_2-\abs{S_1}-\abs{S_2}$. Note that the sets $S_1\cup S_2$ obtained for different choices of $i\in[t]$ are disjoint.

\begin{figure}[ht]
\begin{tikzpicture}
\draw (0,0) arc [start angle=-90, end angle=90, radius=0.6];
\draw (0,0) arc [start angle=270, end angle=90, y radius=0.6, x radius=0.5];
\filldraw (0,0) circle (0.05);
\filldraw (0,1.2) circle (0.05);

\draw (0,1.2) arc [start angle=180, end angle=90, radius=1];

\draw (1,2.2) arc [start angle=-180, end angle=0, radius=0.6];
\draw (1,2.2) arc [start angle=180, end angle=0, x radius=0.6, y radius=0.5];
\filldraw (1,2.2) circle (0.05);
\filldraw (2.2,2.2) circle (0.05);

\draw (2.2,2.2) -- (3.4,2.2);

\draw (3.4,2.2) arc [start angle=-180, end angle=0, radius=0.6];
\draw (3.4,2.2) arc [start angle=180, end angle=0, x radius=0.6, y radius=0.5];
\filldraw (3.4,2.2) circle (0.05);
\filldraw (4.6,2.2) circle (0.05);

\draw (4.6,2.2) -- (5.8,2.2);

\draw (5.8,2.2) arc [start angle=-180, end angle=0, radius=0.6];
\draw (5.8,2.2) arc [start angle=180, end angle=0, x radius=0.6, y radius=0.5];
\filldraw (5.8,2.2) circle (0.05);
\filldraw (7,2.2) circle (0.05);

\draw (7,2.2) -- (8.2,2.2);

\draw (8.2,2.2) arc [start angle=-180, end angle=0, radius=0.6];
\draw (8.2,2.2) arc [start angle=180, end angle=0, x radius=0.6, y radius=0.5];
\filldraw (8.2,2.2) circle (0.05);
\filldraw (9.4,2.2) circle (0.05);

\draw (9.4,2.2) arc [start angle=90, end angle=0, radius=1];

\draw (10.4,0) arc [start angle=270, end angle=90, radius=0.6];
\draw (10.4,0) arc [start angle=-90, end angle=90, y radius=0.6, x radius=0.5];
\filldraw (10.4,0) circle (0.05);
\filldraw (10.4,1.2) circle (0.05);

\draw (0,0) arc [start angle=-180, end angle=-90, radius=1];

\draw[dashed,red,ultra thick] (2.2,-1) -- (3.4,-1);

\draw (1,-1) arc [start angle=-180, end angle=-90, x radius=0.6, y radius=0.5];
\draw[dashed] (1.6,-1.5) arc [start angle=-90, end angle=0, x radius=0.6, y radius=0.5];
\draw[dashed] (1,-1) arc [start angle=180, end angle=100, radius=0.6];
\draw[dashed,red,ultra thick] (2.2,-1) arc [start angle=0, end angle=100, radius=0.6] coordinate (a1);
\draw[dashed,red,ultra thick] (2.2,-1) arc [start angle=0, end angle=80, radius=0.6] coordinate (b1);
\filldraw (1,-1) circle (0.05);
\filldraw (2.2,-1) circle (0.05);

\draw[dashed,red,ultra thick] (4.6,-1) -- (5.8,-1);

\draw[dashed,red,ultra thick] (3.4,-1) arc [start angle=180, end angle=0, radius=0.6];
\draw[dashed] (3.4,-1) arc [start angle=-180, end angle=0, x radius=0.6, y radius=0.5];
\filldraw (3.4,-1) circle (0.05);
\filldraw (4.6,-1) circle (0.05);

\draw (6.4,-1.5) arc [start angle=-90, end angle=0, x radius=0.6, y radius=0.5];
\draw[dashed] (5.8,-1) arc [start angle=-180, end angle=-90, x radius=0.6, y radius=0.5];
\draw[dashed,red,ultra thick] (5.8,-1) arc [start angle=180, end angle=80, radius=0.6];
\draw[dashed] (7,-1) arc [start angle=0, end angle=80, radius=0.6] coordinate (a2);
\filldraw (5.8,-1) circle (0.05);
\filldraw (7,-1) circle (0.05);

\draw (7,-1) -- (8.2,-1);

\draw (8.2,-1) arc [start angle=180, end angle=0, radius=0.6]node [midway, anchor=south] {$+$};
\draw (8.2,-1) arc [start angle=-180, end angle=0, x radius=0.6, y radius=0.5] node [midway,anchor=north] {$-$};
\filldraw (8.2,-1) circle (0.05);
\filldraw (9.4,-1) circle (0.05);

\draw (9.4,-1) arc [start angle=-90, end angle=0, radius=1];

\node at (5.2,0.6) {$F_1$};


\draw (1.2,-6) arc [start angle=270, end angle=90, radius=0.6];
\draw (1.2,-6) arc [start angle=-90, end angle=90, y radius=0.6, x radius=0.5];
\filldraw (1.2,-6) circle (0.05);
\filldraw (1.2,-4.8) circle (0.05);

\draw (1.2,-4.8) arc [start angle=180, end angle=90, radius=1];

\draw (2.2,-3.8) arc [start angle=180, end angle=0, radius=0.6];
\draw (2.2,-3.8) arc [start angle=-180, end angle=0, x radius=0.6, y radius=0.5];
\filldraw (2.2,-3.8) circle (0.05);
\filldraw (3.4,-3.8) circle (0.05);

\draw (3.4,-3.8) -- (4,-3.8);
\draw[dashed,red,ultra thick] (4,-3.8) -- (4.6,-3.8);

\draw[dashed,red,ultra thick] (4.6,-3.8) arc [start angle=180, end angle=0, radius=0.6];
\draw[dashed] (4.6,-3.8) arc [start angle=-180, end angle=0, x radius=0.6, y radius=0.5];
\filldraw (4.6,-3.8) circle (0.05);
\filldraw (5.8,-3.8) circle (0.05);

\draw[dashed,red,ultra thick] (5.8,-3.8) -- (7,-3.8);

\draw[dashed] (7,-3.8) arc [start angle=-180, end angle=0, x radius=0.6, y radius=0.5] coordinate (a);
\draw[dashed,red,ultra thick] (7,-3.8) arc [start angle=180, end angle=90, radius=0.6];
\draw (7.6,-3.2) arc [start angle=90, end angle=0, radius=0.6];
\filldraw (7,-3.8) circle (0.05);
\filldraw (8.2,-3.8) circle (0.05);

\draw (8.2,-3.8) arc [start angle=90, end angle=0, radius=1];

\draw (9.2,-6) arc [start angle=-90, end angle=90, radius=0.6] node [midway, anchor=west] {$+$};
\draw (9.2,-6) arc [start angle=270, end angle=90, y radius=0.6, x radius=0.5] node [midway,anchor=east] {$-$};
\filldraw (9.2,-6) circle (0.05);
\filldraw (9.2,-4.8) circle (0.05);

\draw (1.2,-6) arc [start angle=-180, end angle=-90, radius=1];

\draw (2.2,-7) arc [start angle=-180, end angle=0, radius=0.6];
\draw (2.2,-7) arc [start angle=180, end angle=0, x radius=0.6, y radius=0.5];
\filldraw (2.2,-7) circle (0.05);
\filldraw (3.4,-7) circle (0.05);

\draw (3.4,-7) -- (4.6,-7);

\draw (4.6,-7) arc [start angle=-180, end angle=0, radius=0.6];
\draw (4.6,-7) arc [start angle=180, end angle=0, x radius=0.6, y radius=0.5];
\filldraw (4.6,-7) circle (0.05);
\filldraw (5.8,-7) circle (0.05);

\draw (5.8,-7) -- (7,-7);

\draw (7,-7) arc [start angle=-180, end angle=0, radius=0.6];
\draw (7,-7) arc [start angle=180, end angle=0, x radius=0.6, y radius=0.5];
\filldraw (7,-7) circle (0.05);
\filldraw (8.2,-7) circle (0.05);

\draw (8.2,-7) arc [start angle=-90, end angle=0, radius=1];

\node at (5.2,-5.4) {$F_2$};


\filldraw (4,-3.8) circle (0.05) node [anchor=north] {$y_{i+1}$};
\filldraw (7.6,-3.2) circle (0.05) node [anchor=north] {$y_{i}$};
\filldraw (6.4,-1.5) circle (0.05) node [anchor=south] {$x_{i}$};
\filldraw (1.6,-1.5) circle (0.05) node [anchor=north east] {$x_{i+1}$};

\draw (6.4,-1.5) .. controls (6.4,-2.1) and (7.6,-2.9) .. (7.6,-3.2) node [midway, anchor=south west] {$P_i$};
\draw (1.6,-1.5) .. controls (1.6,-2.8) and (4.2,-2.8) .. (4,-3.8) node [midway, anchor=south west] {$P_{i+1}$};

\draw (a1) circle (0.05);
\filldraw (b1) circle (0.05) node [anchor=south west] {$x'_{i+1}$};
\filldraw (a2) circle (0.05) node [anchor=south west] {$x'_{i}$};
\end{tikzpicture}
\caption{Creating a new adjuster using paths $P_i$ and $P_{i+1}$ with label $(-,+)$. 
Dashed sections are not part of the new adjuster, which has $5$ fewer short cycles. 
The red section shows how much length is lost from $F_1$ and $F_2$; 
in the former case this includes one extra edge since $x_{i+1}$ is on the short side of a cycle.}\label{fig:gadget-merge}
\end{figure}
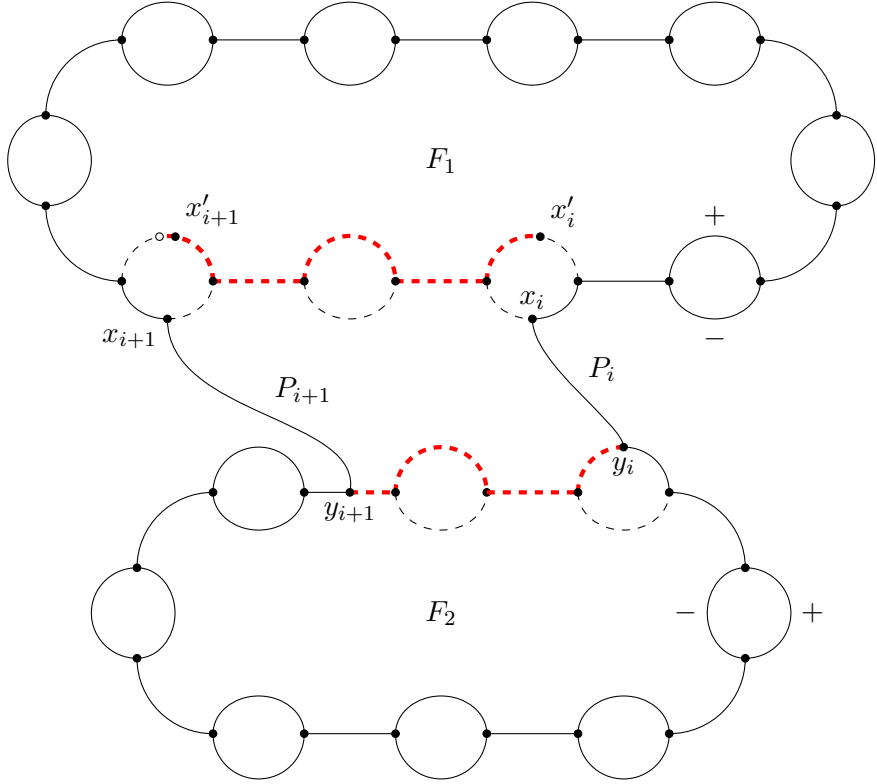

Now we claim that for some pair $P_i,P_{i+1}$ the adjuster constructed has the required properties. 
Note that, since each path $P_i$ has length at most $t$, the upper bound on length is satisfied for all adjusters constructed in this way, 
so it suffices to prove the lower bounds on $\ell$ and $r$. Suppose not, then for each $i$ either more than 
$(r_1+r_2)4s^{-1/2}+4$ short cycles are lost, or $\abs{S_1}+\abs{S_2}>4s^{-1/2}(\ell_1+\ell_2)$. 
In the former case this means more than $(r_1+r_2)4s^{-1/2}$ short cycles lie entirely in the section 
between $x_i$ and $x_{i+1}$ or the section between $y_{i+1}$ and $y_i$, since only $4$ can lie partially in these sections. 
Since these sections are disjoint, the former case occurs for fewer than $s^{1/2}/4$ choices of $i$. 
Similarly the latter case occurs for fewer than $s^{1/2}/4$ choices of $i$. Since there are at least $s^{1/2}/2$ 
choices for $i$, some choice gives an adjuster with the required properties.

Finally, note that the section of the adjuster obtained consisting of vertices not in $F_1$ is contiguous, and contributes at least $\ell-\ell_1$ to the overall length. If $r_2=0$ this section has no short cycles, as required.
\end{proof}

We are now ready to prove the main result of this section.

\begin{proof}[Proof of Lemma~\ref{lem:sn}]
Let $C= 10^{30}$.
We will prove both parts of the statement simultaneously, by finding an $r$-adjuster of suitable length for some $r\geq 9mk/8$. If $s\geq 2$ this adjuster can be taken to have length slightly larger than $n$, such that one of the routes has length exactly $n$. However, for $s=1$ this is not possible, since $G$ itself could have order less than $n$. In this case we show that the length of the adjuster can be made significantly larger than $n/2$.

We repeatedly apply Lemma~\ref{lem:find-gadget}, with $H_{\ref*{lem:find-gadget}}=H'$ and $G_{\ref*{lem:find-gadget}}=G-X_{i-1}$, where $X_0=\varnothing$ and $X_i=V(F'_{i})$ for $i\geq 1$ is the set of vertices used up from $G$ (see the construction of $F_i'$, for each $i$, below), in order to
find $f \leq 4\times 10^4\log^2 k$ adjusters $F_i$ for $i\in[f]$, such that $F_i$ is an $r_i$-adjuster and $r_{\mathrm{total}}:=\sum^{f}_{i=1}r_i\geq 2mk$. Note that, since $k^{22}\geq m\geq C= 10^{30}$ and $k\geq 10\log^4 k$ for any $k\geq 2$, each $F_i$ satisfies
\begin{equation}\abs{F_i}\geq mk\geq m_2+4.1\times 10^{18}\log^4k.\label{fi-bound}\end{equation}
Also, the total size of all these adjusters (using the fact that 
every adjuster found has size at most $mk+3\times 10^4\log^4 k\leq 1.1mk$) is 
at most $5\times 10^4mk\log^2k$.

Set $F'_1=F_1$ and $r'_1=r_1$. After finding each $F_i$ (for $i\geq 2$), we merge it with the adjuster $F'_{i-1}$ using Lemma~\ref{lem:merge-gadgets}. 
There are at least $q=4.1\times 10^{18}\log^4k$ vertex-disjoint paths between the two adjusters, by \eqref{fi-bound} and the hypothesis, and we may choose these to 
have length at most $\abs{H'}+s\leq mk+k$ since the fact that $\overline G$ is $H'$-free implies that any longer path 
contains a short-cut. Thus we obtain a merged adjuster $F'_i$, which is an $r'_i$-adjuster with
\begin{equation}(r'_{i-1}+r_i)\geq r'_i\geq (r'_{i-1}+r_i)(1-4q^{-1/2}) - 4,\label{merge-power}\end{equation}
and 
\begin{equation}\abs{F'_i}\leq \abs{F'_{i-1}}+\abs{F_i}+2mk+2k.\label{merge-total}\end{equation}
It follows from \eqref{merge-total} that for each $i \in [f]$,  \[\abs{F'_{i}}\leq \sum_{j\leq i}\abs{F_{j}}+(i-1)(2mk+2k)\leq (5\times 10^4)mk\log^2k+(4\times 10^4)\log^2 k(2mk+2k)\leq n/10,\]
and so, setting $X_i=V(F'_i)$, we have $\abs{G-X_{i-1}}\geq (sn - n/10) - n/10 \geq sn/10$. Consequently, we can indeed repeatedly apply Lemma~\ref{lem:find-gadget} with $G_{\ref*{lem:find-gadget}} = G-X_{i-1}$ for all $i \in [f]$. 

Using \eqref{merge-power}, by induction we have $r'_{i-1}\leq \sum_{j<i}r_j$ for each $i\geq 2$, and so \begin{equation}r'_i\geq r'_{i-1}+r_i-4q^{-1/2}r_{\mathrm{total}} - 4.\label{r-bound}\end{equation} 
Summing \eqref{r-bound}, we obtain $r'_{f}\geq r_{\mathrm{total}}(1-((16\times 10^4)q^{-1/2}\log^2 k))-4f\geq 3r_{\mathrm{total}}/4$.

We repeatedly apply Lemma~\ref{lem:long-cycle} with $H_{\ref*{lem:long-cycle}} = H'$ and $G_{\ref*{lem:long-cycle}} = G - Y_j$, where $Y_1 = V(F_{f}')$ and $Y_j =Y_1\cup V(C_j')$ for $j \geq 2$ is the set of vertices used up from $G$ so far (see the construction of $C_i'$ below), in order to find $c \leq 1.6 \times 10^6\log^2 k$ cycles $C_i$ each of length at least $n/(8\times 10^5\log^2k) \geq m_2 + 4.1 \times 10^{18}\log^4k$ (which we regard as $0$-adjusters) satisfying the following bound on their total length:
\begin{itemize}\item if $s\geq 2$ then $3n/2 \geq \sum_{i=1}^c\abs{C_i} \geq 4n/3$; whereas
\item if $s=1$ then $3n/4 \geq \sum_{i=1}^c\abs{C_i} \geq 2n/3$.
\end{itemize}

Set $C_1' = C_1$. After finding each $C_i$ (for $i \geq 2$), we merge it with the cycle $C_{i-1}'$ using Lemma~\ref{lem:merge-gadgets}. As before, there are at least $q = 4.1 \times 10^{18}\log^4 k$ vertex-disjoint paths between the two cycles, by the hypothesis, and we may choose these to have length at most $\abs{H'} + s \leq mk + k$. Thus we obtain a merged cycle $C_i'$ with \begin{equation}
     (\abs{C_{i-1}'} + \abs{C_i})(1 - 4q^{-1/2})\leq \abs{C_i'} \leq \abs{C_{i-1}'} + \abs{C_i} +2mk + 2k.
\end{equation} By induction, and that after each merging of cycles we can truncate the length of the new cycle (just as we did with the paths), we have $\abs{C_{i-1}'} \leq \left(\sum_{j<i}\abs{C_j}\right) + 2mk+2m$ for each $i \geq 2$, and so \begin{equation}\label{eq:C_i'}\abs{C_i'} \geq \abs{C_{i-1}'} + \abs{C_i} - 4q^{-1/2}\left(\left(\sum_{j=1}^{c}\abs{C_j}\right) + 2mk + 2m\right).\end{equation} Summing \eqref{eq:C_i'}, we obtain \[\abs{C_c'} \geq \sum_{j=1}^{c}\abs{C_j}(1 - (6.4 \times 10^6)q^{-1/2}\log^2 k) - ((6.4 \times 10^6)q^{-1/2}\log^2 k)(2mk + 2m),\] hence for $s\geq 2$, $\abs{C_c'} \geq 5n/4$ and for $s=1$, $\abs{C_c'} \geq 0.62n$. 

Now merge $C_c'$ with $F_{f}'$ using Lemma~\ref{lem:merge-gadgets} to produce our desired $r$-adjuster. One can see that $r \geq (3/4)^2r_{\mathrm{total}}\geq 9mk/8\geq (m+1)k$.

Suppose $s\geq 2$. Note that the final use of Lemma \ref{lem:merge-gadgets} merges a adjuster of length less than $n$ with a $0$-adjuster of length at least $5n/4$, and thus our $r$-adjuster of length $\ell>n$ contains a section of length at least $\ell-n$ in which there are no short cycles. Using the fact that $\overline G$ is $H'$-free we can shortcut parts of this section, 
if necessary, without reducing $r$, until $0\leq\ell-n\leq mk+k$. 
Now, using the fact that it is an $r$-adjuster for some $r\geq (m+1)k$, we may find a route of length exactly $n$.

Finally, if $s=1$ the final adjuster similarly has length $\ell\geq 0.6 n$ and contains a section of length $\ell-0.6n$ with no short cycles. We can shorten this section, if necessary, as before to obtain an $r$-adjuster of length between $0.6n$ and $0.7n$.
\end{proof}

\section{Stability and exactness}\label{sec:stability}
In this section we prove our main result via the following statement.

\begin{theorem}\label{thm:main2}
There exists a constant $C' > 0$ such that for any complete $k$-partite graph $H$ on $mk$ vertices where $C' \leq m \leq k^{22}$ and $n \geq 10^{60}mk\log^4 k$, we have that $C_n$ is $H$-good, \ie $R(C_n, H) = (\chi(H)-1)(n-1)+\sigma(H)$.
\end{theorem}
\begin{proof}[Proof of Theorem \ref{thm:main}]
Set $k=\chi(H)$ and $m=\abs{H}/k$. By adding edges to $H$, if necessary, without changing $\sigma(H)$, we may assume it is a complete $k$-partite graph for $k=\chi(H)$. If $m>k^{22}$ then the same is true of the size of the largest part, and so Corollary \ref{PS-cor} gives the required bound for some constant $C''$ (and in fact the logarithmic term is not needed). If $C'\leq m\leq k^{22}$ then Theorem \ref{thm:main2} gives the required bound with constant $10^{60}$. Finally, if $m<C'$ then by increasing the size of the largest part we may replace $H$ with a graph $H'$ of order at most $C'\abs{H}$ and $\sigma(H')=\sigma(H)$ which satisfies the conditions of Theorem \ref{thm:main2}, giving that $C_n$ is $H'$-good (and hence $H$-good) for $n\geq 10^{60}C'\abs{H}\log^4\chi(H)$.
\end{proof}

To prove Theorem \ref{thm:main2}, we first establish the following stability result.
\begin{lemma}\label{lem:stability}
There exists a constant $C>0$ satisfying the following. Fix $k\geq 2$ and $C\leq m\leq k^{22}$ and $z\geq 0$ and $n\geq 10^{49}km\log^4 k$. Let $H$ be a complete $k$-partite graph with $mk$ vertices, and write $m_1\leq \cdots\leq m_k$ for the sizes of the parts. 
Define $\hat{H}\sqsubset H$ to be the graph obtained by removing a part of size $m_2$. 
Suppose $G$ is a $C_n$-free graph with $\abs{G}\geq(k-1)(n-1)+z$ such that $\overline{G}$ is $H$-free. 
Then at least one of the following holds.
\begin{enumerate}[(i)]
\item\label{stability-i} There exists $G'\subset G$ such that $\overline{G'}$ is $\hat H$-free and $\abs{G'}\geq (k-2)(n-1)+z$.
\item\label{stability-ii} There exist disjoint sets $A_1,\ldots,A_{k-1}$, with no edges between them, such that for each $i$ we have
\begin{itemize}
\item $\abs{A_i}\geq 0.95n$,
\item $\overline{G[A_i]}$ is $K_{m_1,m_2}$-free, and
\item within $G[A_i]$, any two disjoint sets of size at least $m_2+4.1\times 10^{18}\log^4 k$ have at least 
$4.1\times 10^{18}\log^4 k$ disjoint paths between them.
\end{itemize}
\end{enumerate}
\end{lemma}
\begin{proof}
Let $C$ be a constant given in Lemma~\ref{lem:sn}. 
Suppose not. Start with $A=V(G)$ and $S=\varnothing$. If there exists a set $X$ of order at most $4.1\times 10^{18}\log^4k$ 
such that removing $X$ from $G[A]$ increases the number of components of $G[A]$ of order at least $m_2$, remove $X$ from $A$ 
and add it to $S$. Do this for $k$ iterations or until no such set exists. 
At the end of the process we have $\abs{S}\leq 4.1\times 10^{18}k\log^4k$.

Note that any set $B$ of order at least $m_2$ satisfies $\abs{B\cup N(B)}\geq n$, 
else $G[V(G)\setminus(B\cup N(B))]$ satisfies (i). In particular, if $B\subset A$ with $\abs{B}\geq m_2$ then 
\[\abs{B\cup N_{G[A]}(B)}\geq \abs{B\cup N(B)}-\abs{S}\geq n-4.1\times 10^{18}k\log^4 k\geq \abs{H}.\] 
It follows that the total size of all components of $G[A]$ of size less than $m_2$ is less than $m_2$, 
since otherwise there is a union of some of these components of size between $m_2$ and $2m_2$, 
contradicting the fact that sets of this size expand well. It also follows that any other component has order 
at least $n-4.1\times 10^{18}k\log^4 k$; let these components of $G[A]$ be $A_1,\ldots A_t$, where $t\geq 1$ and 
\begin{equation}\sum_{i\in[t]}\abs{A_i}\geq\abs{A}-m_2\geq \abs{G}-\abs{S}-m_2\geq(k-1)n-k-4.1\times 10^{18}k\log^4k-m_2.
\label{Ai-size}\end{equation}
We must have $t\leq k-1$ since otherwise $\overline G$ contains a copy of $H$ obtained by choosing a 
suitable number of vertices from $A_1,\ldots,A_k$. 

In particular, we added vertices to $S$ fewer than $k$ times, and so the process described above stopped 
because no more removals were possible. This will give the required connectivity properties. 
Indeed, if $B_1,B_2$ are disjoint sets in $A_i$ of size at least $m_2+4.1\times 10^{18}\log^4 k$ then no set $X$ 
of order $4.1\times 10^{18}\log^4 k$ separates $B_1\setminus X$ from $B_2\setminus X$ within $G[A_i]$ 
(since both these sets have size at least $m_2$ and by the construction of $A$). 
Thus, by Menger's theorem there are at least $4.1\times 10^{18}\log^4 k$ disjoint paths between them in $G[A_i]$.

For each $i\leq t$, we define $k-t+1$ subgraphs $H^{(1)}_i\sqsubset \cdots \sqsubset H^{(k-t+1)}_i\sqsubset H$ as follows. 
Set $H_i^{(1)}$ to be an independent set of $m_{k+1-i}$ vertices, and for $2\leq j\leq k-t+1$ set 
$H_i^{(j)}=K_{m_{k+1-i},m_{k-t},m_{k-t-1},\ldots,m_{k-t-j+2}}$, \ie $H_i^{(j)}$ is a complete $j$-partite subgraph consisting of 
the $i$th largest class of $H$ together with $j-1$ of the $k-t$ smallest classes taken in decreasing order. 

For each $i\leq t$, let $s_i\geq 1$ be the largest value such that $\overline{G[A_i]}$ contains $H_i^{(s_i)}$. 
Suppose that $\sum_{i\in[t]} s_i\geq k$, and consider the graph induced in $\overline{G}$ by the vertices of a copy of 
$H_i^{(s_i)}$ in $\overline{G[A_i]}$ for each $i$. Taking the largest class from each copy creates a copy of 
$K_{m_k,\ldots,m_{k-t+1}}$, and the other classes give a complete $(k-t)$-partite graph with $j$th largest part 
having size at least $m_{k-t+1-j}$ for each $j\in[k-t]$, which contains $K_{m_{k-t},\ldots,m_1}$. Thus $\overline G\supset H$, a contradiction.

Consequently we have $\sum_{i\in[t]} s_i\leq k-1$, and in particular $s_i\leq k-t$ for each $i$. By definition of $s_i$, 
$\overline{G[A_i]}$ is $H^{(s_i+1)}_i$-free for each $i$. If $s_i>1$ we have $\abs{A_i}<s_in-n/10$, since otherwise Lemma~\ref{lem:sn}, taking $H'_{\ref*{lem:sn}}=H_i^{(s_i+1)}$, gives a copy of $C_n$ inside $G[A_i]$.
Contrariwise, if $s_i=1$ we have $\abs{A_i}<n+m_{k-t}$ by Corollary \ref{PS-cor38}. In either case, for each $i\in [t]$ we have $\abs{A_i}<s_in+m_{k-t}$. Note that $(t-1)m_{k-t}\leq\abs{H}=mk$. 
Thus, if for any $i$ we have $s_i>1$ then \[\sum_{j\in [t]}\abs{A_j}<s_in-n/10+\sum_{\substack{j\in[t]\\j\neq i}}(s_in+m_{k-t})\leq (k-1)n-n/10+(t-1)m_{k-t},\] contradicting \eqref{Ai-size}. We may therefore assume each $s_i=1$. Similarly if $\sum_{i\in [t]} s_i<k-1$, and so $t<k-1$, we obtain a contradiction to \eqref{Ai-size}. Thus $t=k-1$, and each of $A_1,\ldots,A_{k-1}$ has size at least 
$n-4.1\times 10^{18}k\log^4k\geq 0.95n$. If for some $i$ we have $K_{m_1,m_2}\subset\overline{G[A_i]}$, 
then taking the copy of $K_{m_1,m_2}$ together with $m_k$ vertices from each other part gives a copy of 
$K_{m_k,\ldots,m_k,m_2,m_1}$ (with $k$ parts in total) in $\overline{G}$, and so $\overline G\supset H$, 
a contradiction. Together with the connectivity properties heretofore established, this completes the proof.
\end{proof}

We now proceed from the stability result above to exactness. 
We need the following result which says that when the complement is $K_{m_1,m_2}$-free we can do much better than Lemma~\ref{lem:merge-gadgets}.

\begin{lemma}\label{lem:merge-efficient}Let $F_i$ be an $r_i$-adjuster of length $\ell_i$ for $i\in\{1,2\}$, 
with $F_1$ and $F_2$ disjoint, and suppose $\overline{G[V(F_i)]}$ is $K_{m_1,m_2}$-free for each $i$, where $m_2\geq m_1$. 
Suppose further that there are two vertex-disjoint paths $P_1,P_2$ between $F_1$ and $F_2$, each of length at most $t$. 
Then there is an $r$-adjuster of length $\ell$ all of whose vertices are contained in $F_1\cup F_2\cup P_1\cup P_2$, 
where $r\geq r_1+r_2-2(m_1+m_2)$ and $\ell_1+\ell_2-2(m_1+m_2)\leq\ell\leq\ell_1+\ell_2+2t$.
Furthermore, if $r_2=0$ then the adjuster contains a section of length at least $\ell-\ell_1$ with no short cycles.
\end{lemma}
\begin{proof}Let $P_i$ have ends $x_i$ in $F_1$ and $y_i$ in $F_2$. We claim that there is a path $Q_1$ of length at least 
$\ell_1-m_1-m_2-1$ in $F_1$ between $x_1$ and $x_2$. 
To see this, first note that unless $x_1$ and $x_2$ are on opposite sides of the same short cycle, there is a route in $F_1$ 
which includes both $x_1$ and $x_2$, and has length at least $\ell_1-2$. If $x_1$ and $x_2$ are at distance at most $m_2$ 
in this route then by taking the whole route except for a path of length at most $m_2$ between the two vertices, we are done. 
If not, then the $m_1$ vertices immediately before $x_1$ on the route do not include $x_2$, and the $m_2$ vertices immediately 
before $x_2$ do not include $x_1$. There is an edge $z_1z_2$ between these two sets of vertices in $G[V(F_1)]$, since $\overline{G[V(F_1)]}$ is $K_{m_1,m_2}$-free, and so there 
is a path which starts at $x_1$, goes around the route to $z_2$, via the extra edge to $z_1$ and then around the route in 
the opposite direction to $x_2$. This path misses out at most $m_1+m_2-1$ vertices from the route, so has at least the required length. See Figure \ref{fig:nearly-covering-path}.
Finally, if $x_1$ and $x_2$ are on opposite sides of the same short cycle, there is a path around the adjuster from $x_1$ to $x_2$ using the
vertices on that short cycle after $x_1$ and before $x_2$, and the longer side of every other short cycle. Similarly there is
another path which uses the vertices on that short cycle after $x_2$ and before $x_1$. At least one of these uses at least half of every
short cycle, so has length at least $\ell_1$.

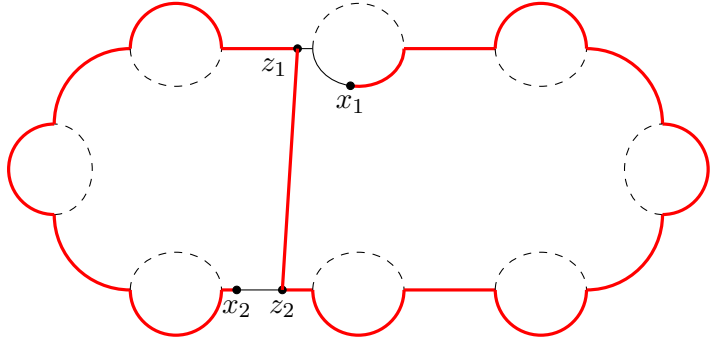
\begin{figure}
    \centering
    \begin{tikzpicture}
\draw[very thick, red] (1.2,-6) arc [start angle=270, end angle=90, radius=0.6];
\draw[dashed] (1.2,-6) arc [start angle=-90, end angle=90, y radius=0.6, x radius=0.5];

\draw[very thick, red] (1.2,-4.8) arc [start angle=180, end angle=90, radius=1];

\draw[very thick, red] (2.2,-3.8) arc [start angle=180, end angle=0, radius=0.6];
\draw[dashed] (2.2,-3.8) arc [start angle=-180, end angle=0, x radius=0.6, y radius=0.5];

\draw[very thick, red] (3.4,-3.8) -- (4.4,-3.8);
\draw (4.4,-3.8) -- (4.6,-3.8);

\draw[dashed] (4.6,-3.8) arc [start angle=180, end angle=0, radius=0.6];
\draw (4.6,-3.8) arc [start angle=-180, end angle=-100, x radius=0.6, y radius=0.5] coordinate (x1);
\draw[very thick, red] (x1) arc [start angle=-100, end angle=0, x radius=0.6, y radius=0.5];
\filldraw (x1) circle (0.05) node[anchor=north] {$x_1$};
\filldraw (4.4,-3.8) circle (0.05) node[anchor=north east] {$z_1$};

\draw[very thick, red] (5.8,-3.8) -- (7,-3.8);

\draw[dashed] (7,-3.8) arc [start angle=-180, end angle=0, x radius=0.6, y radius=0.5] coordinate (a);
\draw[very thick, red] (7,-3.8) arc [start angle=180, end angle=0, radius=0.6];

\draw[very thick, red] (8.2,-3.8) arc [start angle=90, end angle=0, radius=1];

\draw[very thick, red] (9.2,-6) arc [start angle=-90, end angle=90, radius=0.6];
\draw[dashed] (9.2,-6) arc [start angle=270, end angle=90, y radius=0.6, x radius=0.5];

\draw[very thick, red] (1.2,-6) arc [start angle=-180, end angle=-90, radius=1];

\draw[very thick, red] (2.2,-7) arc [start angle=-180, end angle=0, radius=0.6];
\draw[dashed] (2.2,-7) arc [start angle=180, end angle=0, x radius=0.6, y radius=0.5];

\draw (3.6,-7) -- (4.2,-7);
\draw[very thick, red] (3.4,-7) -- (3.6,-7);
\draw[very thick, red] (4.2,-7) -- (4.6,-7);
\filldraw (3.6,-7) circle (0.05) node[anchor=north] {$x_2$};
\filldraw (4.2,-7) circle (0.05) node[anchor=north] {$z_2$};

\draw[very thick, red] (4.4,-3.8) -- (4.2,-7);

\draw[very thick, red] (4.6,-7) arc [start angle=-180, end angle=0, radius=0.6];
\draw[dashed] (4.6,-7) arc [start angle=180, end angle=0, x radius=0.6, y radius=0.5];

\draw[very thick, red] (5.8,-7) -- (7,-7);

\draw[very thick, red] (7,-7) arc [start angle=-180, end angle=0, radius=0.6];
\draw[dashed] (7,-7) arc [start angle=180, end angle=0, x radius=0.6, y radius=0.5];

\draw[very thick, red] (8.2,-7) arc [start angle=-90, end angle=0, radius=1];
    \end{tikzpicture}
    \caption{A long $x_1$-$x_2$ path using the extra edge $z_1z_2$ (shown in bold, and highlighted if colour is available). Dashed lines indicate parts of the adjuster not in the relevant route.}
    \label{fig:nearly-covering-path}
\end{figure}

Now taking this path together with a corresponding path for $F_2$ and the two paths $P_1,P_2$ between them gives a cycle $C$ of length 
between $\ell_1+\ell_2-2(m_1+m_2)$ and $\ell_1+\ell_2+2t$. We extend this to an adjuster by including all short cycles from $F_1$ 
and $F_2$ such that one of the routes round that short cycle is entirely contained in $C$. Since all but at most $m_1+m_2$ 
edges of a route of $F_i$ are included in the cycle, this means at most $2(m_1+m_2)$ short cycles of $F_1$ or $F_2$ are not included, giving $r\geq r_1+r_2-2(m_1+m_2)$. Finally, note that the section of the adjuster obtained consisting of vertices not in $F_1$ is contiguous, and contributes at least $\ell-\ell_1$ to the overall length. If $r_2=0$ this section has no short cycles, as required.
\end{proof}
This enables us to effectively deal with the case where two of the $A_i$ (or, more precisely, the expanding subgraphs within them) 
are connected by two disjoint paths. 

\begin{proof}[Proof of Theorem~\ref{thm:main2}] 
It suffices to prove, via induction on $k$, the conclusion for $2C\frac{k-1}{k}\le m\le k^{22}$, where $C$ is at least the constants in Lemma~\ref{lem:sn} and Lemma~\ref{lem:stability}. Since we have $C\leq 2C\frac{k-1}{k}\leq 2C$, we can then take $C'=2C$ to conclude. 

Suppose it is not true, and take a counterexample, \ie a graph $G$ of order $(k-1)(n-1)+\sigma(H)$ that is $C_n$-free 
and with $H$-free complement. We apply Lemma~\ref{lem:stability}. If \eqref{stability-i} applies then we have a graph 
$G'\subset G$ of order $(k-2)(n-1)+\sigma(H)$, with $\hat H$-free complement. Note that by definition of $\hat H$ we have 
$\sigma(H)=\sigma(\hat H)$. Also, since $m_2\leq \frac{mk}{k-1}$, we have 
\[m(\hat H)\geq m\frac{k(k-2)}{(k-1)^2}\geq 2C\frac{k-2}{k-1}.\] 
Thus if $\abs{\hat H}\leq (k-1)^{23}$ we have $G'\supset C_n$ by the induction hypothesis, 
whereas otherwise we have $G'\supset C_n$ by Corollary~\ref{PS-cor}. Consequently we must have \eqref{stability-ii}.

Within each set $A_i$ we apply Lemma \ref{lem:expansion} with $G_{\ref*{lem:expansion}}=G[A_i]$, $H_{\ref*{lem:expansion}}=K_{m_1,m_2}$,  $M_{\ref*{lem:expansion}}=\abs{A_i}/((m_1+m_2)\log 2)$ and $\beta_{\ref*{lem:expansion}}=M_{\ref*{lem:expansion}}/100>40$ (note that also $\beta_{\ref*{lem:expansion}} \geq 10 \log k$) to find a $(10,M_{\ref*{lem:expansion}}/100,(m_1+m_2)/2,2)$-expander $H_i$ of order at least $\abs{A_i}-(m_1+m_2)/2\geq \abs{A_i}-m$.

Suppose that $H_i$ and $H_j$ are linked (in $G$) by two disjoint paths for some distinct $i,j\in[k-1]$. Recall that $\abs{A_i}\geq 0.95n$. Applying Lemma~\ref{lem:sn} to $G[A_i]$, we find an $r$-adjuster $F_i$ for $r\geq 9mk/8$ with length between $0.6n$ and $0.7n$. 
Similarly, we may find an adjuster in $A_j$ of length between $0.6n$ and $0.7n$; however, instead of using it as an adjuster we simply take the longest route to find a cycle $C_j$ of that length. 

Note that the size of the adjuster and cycle ensures that they intersect $H_i$ and $H_j$ respectively in at least $\abs{H_i}+0.6n-\abs{A_i}\geq 0.5n$ vertices. Suppose that the disjoint paths between $H_i$ and $H_j$ have endpoints $x_1$ and $x_2$ in $H_i$ and endpoints $y_1$ and $y_2$ in $H_j$.
By Lemma \ref{lem:wings}, we may find disjoint sets $B_1\ni x_1$ and $B_2\ni x_2$ in $V(H_i)$ of size $\abs{A_i}/200$, such that every vertex in $B_1$ is connected to $x_1$ by a path, and likewise for $B_2$ and $x_2$. 
We may therefore use Lemma \ref{lem:short-path} with $A_{\ref*{lem:short-path}}=V(F_i)\cap V(H_i),B_{\ref*{lem:short-path}}=B_1\cup B_2,C_{\ref*{lem:short-path}}=\varnothing$ to give a path from  $V(F_i)$ to $B_1$ or $B_2$, without loss of generality the former.
Extend it to a path from $V(F_1)$ to $x_1$ using vertices from $B_1$, and remove vertices from this path, if necessary, to obtain a shortest path $P_1$. This ensures that $\abs{V(P_1)}\leq m_1+m_2+1$ since $\overline{G[A_i]}$ is $K_{m_1,m_2}$-free; note also that $P_1$ is disjoint from $B_2$.
Next apply Lemma \ref{lem:short-path} with $A_{\ref*{lem:short-path}}=(V(F_i)\cap V(H_i))\setminus V(P_1),B_{\ref*{lem:short-path}}=B_2,C_{\ref*{lem:short-path}}=V(P_1)$ to give a disjoint $x_2$-$V(F_i)$ path $P_2$.
We can similarly find two disjoint paths from $y_1$, $y_2$ to $F_j$ in $H_j$, and the union of these with the existing paths between $H_i$ and $H_j$ give two disjoint paths between the adjuster and the cycle.
Since $\overline{G}$ is $H$-free, each of these paths contains a shortest path of length at most $mk+k$. We apply Lemma~\ref{lem:merge-efficient} to merge the adjuster and cycle using these shortest paths, obtaining
an adjuster with at least $9mk/8-2(m_1+m_2)\geq mk+k$ short cycles, of total length $\ell$ between $1.2n-2(m_1+m_2)> 1.19n$ and $1.4n+2(mk+k)<1.41n$. Recall that $F_i$ has length at most $0.7n$, so the adjuster obtained from merging $F_i$ and $C_j$ contains a section of length at least $\ell-0.7n\ge 0.49n>\ell-n$ without short cycles (the part corresponding to $C_j$). We may reduce this section until the total 
length is between $n$ and $n+mk+k$, since $\overline G$ is $H$-free, and then use the reducing power of the adjuster to obtain a cycle of length exactly $n$.

Thus we may assume that no pair $H_i,H_j$ is linked by two disjoint paths. It follows that for some index $i$ there is 
a vertex $v$ which separates $H_i$ from all other $H_j$. Let $A$ be the component of $G-v$ containing $H_i-v$ 
and let $G'=G[A\cup\{v\}]$. Note that $\overline{G'}$ is $K_{m_1+1,m_2+1}$-free, since otherwise we may choose $m_k$
vertices from $H_j-v$ for each $j\neq i$ together with the vertices of a $K_{m_1,m_2}$ in $\overline{G'-v}$ 
to give a graph containing $H$ in $\overline{G}$. Recall that \eqref{stability-i} of Lemma~\ref{lem:stability} does not apply.\vspace{0.5mm}
It follows that any set $B\subset A$ of size $m_2$ has $\abs{N_G(B)\cup B}\geq n$, since otherwise $\overline{G-(N_G(B)\cup B)}$ is not $\hat{H}$-free, implying that $\overline G$ is not $H$-free. Since also $N_{G'}(B)\cup B=N_G(B)\cup B$, we have $\abs{N_{G'}(B)\cup B}\geq n$. 
Thus, for any set $B'\subset V(G')$ of order at least $m_2+1$, we have $\abs{N_{G'}(B')\cup B'}\geq n$ (since $B'$ contains a 
set of order $m_2$ not including $v$). Now, using Lemma~\ref{lem:long-cycle} with $G_{\ref*{lem:long-cycle}}=H_i$ and $H_{\ref*{lem:long-cycle}}=K_{m_2,m_2}$, we may find a cycle of length at least $10m_2$ 
in $H_i$. Taking $x,y$ to be two consecutive vertices on this cycle, Lemma \ref{PS-lemma} applies to $G_{\ref*{PS-lemma}}=G'$ with $m_{\ref*{PS-lemma}}=m_2+1$,
and we may find a path of order exactly $n$ between $x$ and $y$, giving a copy of $C_n$.
\end{proof}

\section*{Acknowledgement}
The authors are grateful to the referee for their careful review.

\end{document}